\documentclass[11pt, reqno]{amsart}

\usepackage{amscd,amsmath}
\usepackage{mathrsfs}
\usepackage{amsfonts}
\usepackage{amssymb}
\usepackage{enumerate}
\usepackage{setspace}
\usepackage{color}
\usepackage{bm}
\usepackage{esint}
\usepackage{float}
\usepackage{natbib}

\usepackage{tikz}
\usetikzlibrary{positioning,fadings,through}

\usepackage{anysize}
\marginsize{1.5cm}{1.5cm}{1.5cm}{1.5cm}

\usepackage[english]{babel}

\usepackage{tabu}

\usepackage[colorlinks=true, citecolor=blue]{hyperref}

\newtheorem{theorem}{Theorem}[section]
\newtheorem{lemma}{Lemma}[section]

\newcommand{\eqn}{\begin{eqnarray}}
\newcommand{\een}{\end{eqnarray}}

\newcommand{\ZZ}{\mathbb{Z}}

\newcommand{\pat}{\partial_t}

\definecolor{luh-dark-blue}{rgb}{0.0, 0.313, 0.608}

\newcommand{\bR}{\mathbb{R}}
\newcommand{\bT}{\mathbb{T}}
\newcommand{\bZ}{\mathbb{Z}}

\numberwithin{equation}{section}

\usepackage{scalerel,stackengine}
\stackMath
\newcommand\reallywidehat[1]{%
\savestack{\tmpbox}{\stretchto{%
  \scaleto{%
    \scalerel*[\widthof{\ensuremath{#1}}]{\kern-.6pt\bigwedge\kern-.6pt}%
    {\rule[-\textheight/2]{1ex}{\textheight}}
  }{\textheight}%
}{0.5ex}}%
\stackon[1pt]{#1}{\tmpbox}%
}
\usepackage{mathtools}%

\def\cb{c_{\m{B}}}
\def\cs{c_{\m{S}}}
\newcommand{\m}[1]{{\mathcal{ #1}}}

\newcommand{\w}[1]{{\widehat{ #1}}}

\newcommand{\mS}[1]{{S^{(#1)}}}
\newcommand{\mB}[1]{{B^{(#1)}}}
\newcommand{\mP}[1]{{P^{(#1)}}}
\newcommand{\mQ}[1]{{Q^{(#1)}}}
\newcommand{\mZ}[1]{{Z^{(#1)}}}
\newcommand{\mh}[1]{{h^{(#1)}}}
\newcommand{\mb}[1]{{b^{(#1)}}}

\def\Gt{\Gamma_{top}}
\def\Gb{\Gamma_{bot}}

\def\q{\quad}
\def\qq{\qquad}

\def\Ga{\Gamma}

\def\lm{\lambda}
\def\de{\delta}
\def\si{\sigma}
\def\ga{\gamma}
\def\b{\beta}

\def\eps{\varepsilon}
\def\N{\nabla}
\def\D{\Delta}

\def\tsi{\tilde{\si}}

\let\cosh\relax 
\DeclareMathOperator*{\cosh}{\t{ch}}
\let\sinh\relax 
\DeclareMathOperator*{\sinh}{\t{sh}}
\let\coth\relax 
\DeclareMathOperator*{\coth}{\t{cth}}

\newcommand{\pare}[1]{\left( #1 \right)}
\newcommand{\norm}[1]{\left\| #1 \right\|}
\newcommand{\av}[1]{\left| #1 \right|}
\newcommand{\bra}[1]{\left[ #1 \right]}
\newcommand{\set}[1]{\left\{ #1 \right\}}

\renewcommand{\t}[1]{\text{#1}}

\newcommand{\diff}[1]{\de\hspace{-1mm}\bra{#1}}
\newcommand{\ch}[1]{\cosh\pare{#1}}
\newcommand{\sh}[1]{\sinh\pare{#1}}

\newcommand{\cth}[1]{\coth\pare{#1}}
\allowdisplaybreaks

\begin{document}

\setlength{\abovedisplayskip}{5pt}
\setlength{\belowdisplayskip}{5pt}

\setlength{\jot}{7pt}

\title[Multilayer tumor growth]{A new asymptotic model of multilayer tumor growth}

\author{Rafael Granero-Belinch\'on}
\address{Departamento de Matem\'aticas, Estad\'istica y Computaci\'on, Universidad de Cantabria, Spain.}
\email{rafael.granero@unican.es}

\author{Martina Magliocca}
\address{Departamento de An\'alisis Matem\'atico, Universidad de Sevilla, Spain.}
\email{mmagliocca@us.es}

\date{\today}

\subjclass[2010]{}


\begin{abstract}
In this paper we study the growth of a tumor colony of multilayer type and focus on how the tumor grows from a near flat (when compared to the length of the tumor as, for instance, in the case of a bone tumor in a femur) initial colony. In particular we derive and study a new weakly nonlinear asymptotic model of multilayer tumor growth. The model takes the form of a nonlinear and nonlocal high order system of PDEs. Finally, motivated by the possibility of a finite time collision of the interfaces, we study the well-posedness of this system. 
\end{abstract}

\maketitle

\setcounter{tocdepth}{3}

{\small \tableofcontents}

\section{Introduction}

Cancer is a very complex phenomenon that involves many different scales and situations. One of the easiest tumor models  consists
only of proliferating cancer cells ({\it i.e.}, the tumor expands and so does its boundary),  whose growth is regulated by an externally supplied nutrient, such as oxygen or glucose. The model gets more interesting if we also assume that a certain inhibitor (for instance chemotherapy, drugs...) acts on the proliferating cells, trying to stop their growth. \\
Of course, tumor can be made up of different types of cells, depending on the tumor status. Beyond proliferating cells, we could also consider necrotic {\it i.e.}, dead cells) and quiescent cells ({\it i.e.},  cells that are still alive but stopped proliferating). 

The importance of this illness together with the complexity of the phenomenon lead to a large number of research studies by many different research groups in Biology and Medicine but also in Physics, Mathematics or Engineering. In fact, the first mathematical models related to cancer goes back to the fifties (see \citep{araujo2004history,byrne2010dissecting} for very complete reviews). \\
From a mathematical point of view, the main questions related to tumor models concern well-posedness, bifurcation phenomena, and asymptotic stability. We quote \citep{friedman1999analysis} for a mathematical overview. \\
Also, the mathematical approach to this kind of problems has been both discrete and continuous. We refer, for instance, to  \citep{d2007nonlinear,d2008metamodeling} in the first case, and \citep{wu2020asymptotic,cui2002analysis,zhou2009existence} in the latter.

In real life, tumors grow under physical spatial constraints. Particular examples of this are skin or bone tumors, where the tumor colony initiates its growth in a domain with a specific geometry. For instance, a bone tumor is constrained to a domain whose longitudinal measure may clearly be larger than the rest of the dimensions. This specific geometry could potentially be of great importance in how the tumor develops \citep{fiore2020mechanics}. However, the challenges of tumor growth and its mathematical modelling cause that most of the mathematical results are obtained in a simplified geometrical setting. For instance, many mathematical results are available for what are called multicellular spheroids, \emph{i.e.} when the tumor boundary is a sphere whose radious grows or shrinks.  However, an important point to underline is that tumors {\it in vivo} are not radially symmetric. Indeed,
it is known that the space occupied by the interacting cancer cells possesses fractal structures \citep{d2009fractal}.

In this paper we study the growth of a tumor colony of multilayer type and we will focus on how the tumor grows from a near flat (when compared to the length of the tumor as, for instance, in the case of a bone tumor in a femur) initial colony. In other words, we want to mathematically describe the growth of a long tumor with a small width and see how this particular geometry influences the dynamics. In this regards, we find an asymptotic model that accurately describe the growth of the tumor in the regime where the dimensionless parameter
$$
\frac{\text{width}}{\text{length}}=\varepsilon\ll 1.
$$

In our opinion, one of the main questions in the mathematical description of tumor growth is whether the presence of an upper and a lower tumor boundaries could lead to a finite time split of the tumor colony into two separated sets of oncogenic cells. To the best of our knowledge, such a real life problem has not been mathematically described. In fact, we suspect that it is a really challenging mathematical question that could shed light on the validity of the models that we are using to describe tumor growth. To advance towards an answer in this paper we derive and study an asymptotic model of a multilayer tumor having upper and lower free boundaries.

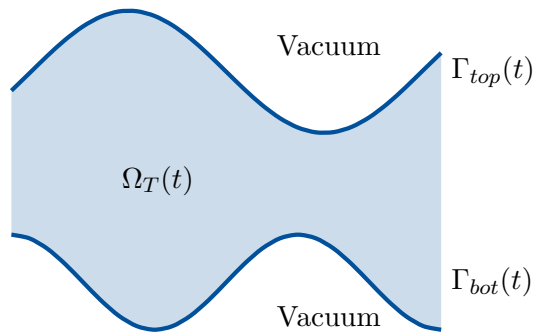
\begin{figure}[H]\label{Fig4}
\begin{tikzpicture}[domain=-pi:pi, scale=0.9]

\shade[top color=luh-dark-blue!20, bottom color=luh-dark-blue!20] plot[domain=-pi:pi] (\x,{2.1-0.9*sin(1.1*\x r)}) -- plot[domain=pi:-pi] (\x,{-1+0.7*sin(1.5*\x r)});

\draw[ultra thick, smooth, variable=\x, luh-dark-blue] plot (\x,{2.1-0.9*sin(1.1*\x r)});

\draw[ultra thick, smooth, variable=\x, luh-dark-blue] plot (\x,{-1+0.7*sin(1.5*\x r)});

\node[right]  at (pi,2.1) {$\Gt(t)$};
\node[right]  at (pi,-1) {$\Gb(t)$};
\node at (1.5,2.5) {Vacuum};
\node at (1.5,-1.5) {Vacuum};
\node  at (-1,0.5) {$\Omega_T(t)$};
\end{tikzpicture}
\caption{The one-phase case with only the tumor colony}
\end{figure}

Such configuration appears in several real life tumor such as the ones in Cristini \& Lowengrub \citep{cristini2010multiscale} or \citep{fiore2020mechanics} and references therein.

We consider a tumor colony described as
\begin{align*}
\Omega_T(t)&=\set{x_1\in L\mathbb{S}^1,\;-L+b(x_1,t)<x_2<h(x_1,t)},
\end{align*}
where $L$ is the typical length of the tumor colony and $L\mathbb{S}^1$ denotes the circle with length 2$L$ (or equivalently the interval $[-L\pi,L\pi]$ with periodic boundary conditions).
The boundary $\Gamma(t)$ of this domain has two components $\Gt(t)$ and $\Gb(t)$:
\begin{align*}
\Gt(t)&=\set{ x \in\mathbb{R}^2, x_1\in L\mathbb{S}^1,\;x_2=h(x_1,t)},\\
\Gb(t)&=\set{ x \in\mathbb{R}^2, x_1\in L\mathbb{S}^1,\;x_2=-L+b(x_1,t)},\\
\Ga(t)&=\Gt(t)\cup\Gb(t).
\end{align*}
In addition, we assume that the tumor dynamics is modelled using the following free boundary problem
\begin{subequations}\label{eq:pb-dim-inf}
\begin{align}
\frac{\partial \sigma}{\partial t}-D_n\Delta \sigma&=\delta_n\left(\sigma_B-\sigma\right)-\lambda_n \sigma -\gamma_n\beta&&\text{ in }\Omega_T(t)\times(0,T),\label{eq:siT-dim-tb}\\
\frac{\partial \beta}{\partial t}-D_i\Delta \beta&=\delta_i\left(\beta_B-\beta\right)-\lambda_i \beta&&\text{ in }\Omega_T(t)\times(0,T),\label{eq:bT-dim-tb}\\
 {u}&=-\nabla p+\chi\nabla \sigma&&\text{ in }\Omega_T(t)\times(0,T),\label{eq:u-dim-tb}\\
\nabla\cdot  {u}&=\mu(\sigma-\tilde{\sigma}-\tau\beta)&&\text{ in }\Omega_T(t)\times(0,T),\label{eq:divu-dim-tb}\\
p&=-\nu\frac{h,_{11}}{\pare{1+h,_1^2}^{3/2}}&&\text{ on }\Gt(t)\times(0,T),\label{eq:p-dim-top}\\
p&=\nu\frac{b,_{11}}{\pare{1+b,_1^2}^{3/2}}&&\text{ on }\Gb(t)\times(0,T),\label{eq:p-dim-bot}\\
\frac{\partial h}{\partial t}&= {u}\cdot (-h,_1,1)&&\text{ on }\Gt(t)\times(0,T),\label{eq:dth-dim-top}\\
\frac{\partial b}{\partial t}&= {u}\cdot (-b,_1,1)&&\text{ on }\Gb(t)\times(0,T),\label{eq:dtb-dim-bot}\\
\si&=\si_D&&\text{ on }\Ga(t)\times(0,T),\\
\b&=\b_D&&\text{ on }\Ga(t)\times(0,T),
\end{align}
\end{subequations}
with the initial data
$$
\si(x,0)=\si_0(x),\qq
\beta(x,0)=\beta_0(x)\t{ in }\Omega(0),\qq
h(x_1,0)=h_0(x_1),\qq
b(x_1,0)=b_0(x_1).
$$
In the previous system $\sigma$ and $\beta$ are the nutrients and inhibitors for the cancer colony, respectively. We assume that these chemicals are consumed by the tumor and diffused through the tumoral tissue. We also assume that the tumor has access to vasculature via the $\delta_{n}$ and $\delta_{i}$ terms. The term $u$ is the \emph{velocity} of the tumor growth. We assume that such a velocity follows a Darcy's law \citep{greenspan1972models,greenspan1974self,greenspan1976growth} with a chemotactic term proportional to $\chi$. The oncogenic cells are assumed to follow a standard Laplace-Young law \citep{greenspan1972models,greenspan1974self,greenspan1976growth}. Finally, $h$ and $b$ are the tumor upper and lower free boundaries. These boundaries are unknown for the problem and must be solved for. In the case of a infinitely deep tumor following the same laws we refer to \citep{granero2025nonlocal}.\\
It has been pointed out to us that the term $-\ga_n\b$ in equation \eqref{eq:siT-dim-tb} lacks realism. However, as in \citep{byrne1995growth}, we keep it for the sake of completeness.

In this paper, for the previous system \eqref{eq:pb-dim-inf} we derive the following system of nonlinear and nonlocal PDEs
\begin{align}\label{eq:patU}
\pat U(x_1,t)
&=-\eta\bra{\Theta_1 \pare{\Lambda^2 U(x_1,t)}+\Theta_2 \pare{\Lambda^2 V(x_1,t) }}\nonumber\\
&\q+\eps\eta
\Theta_1\pare{U(x_1,t)\Theta_1\Lambda^2U(x_1,t)}
+\eps\eta
\Theta_1\pare{U(x_1,t)\Theta_2\Lambda^2V(x_1,t)}\nonumber\\
&\q+\eps\eta\pare{U(x_1,t)\Lambda^2U(x_1,t),_1},_1\nonumber\\
&\q -\eps\eta\Theta_2\pare{V(x_1,t)\Theta_2\Lambda^2U(x_1,t)}-\eps\eta
\Theta_2\pare{V(x_1,t)\Theta_1\Lambda^2V(x_1,t)}\nonumber
\\
&\q -\eps\bra{
\theta\b_1(t)e^{-Nt}+
\frac{\rho}{12}\b_2(t)e^{-Nt}
}
\pare{\Theta_1 U(x_1,t)+\Theta_2 V(x_1,t)}\nonumber\\
&\q +K^{(0)}(x_1,t)+\eps  K^{(1)}(x_1,t),
\end{align}
and
\begin{align}\label{eq:patV}
\pat V(x_1,t)
&=-\eta\bra{
\Theta_2\pare{\Lambda^2U(x_1,t)}+\Theta_1\pare{\Lambda^2V(x_1,t)}
}\nonumber\\
&\q+ \eps\eta\Theta_2\pare{U(x_1,t)\Theta_1\Lambda^2U(x_1,t)}+\eps\eta\Theta_2\pare{U(x_1,t)\Theta_2\Lambda^2V(x_1,t)}\nonumber\\
&\q-\eps\eta\Theta_1\pare{V(x_1,t)\Theta_2\Lambda^2U(x_1,t)}-\eps\eta\Theta_1\pare{V(x_1,t)\Theta_1\Lambda^2V(x_1,t)} \nonumber\\
&\q -\eps\eta \pare{V(x_1,t)\Lambda^2V,_1(x_1,t)},_1\nonumber\\
&\q -\eps\bra{
\theta e^{-Nt}\b_1(t)
+
\frac{\rho}{12} e^{-Nt}\b_2(t)
}\pare{\Theta_2 U(x_1,t)+\Theta_1 V(x_1,t)}\nonumber\\
&\q -K^{(0)}(x_1,t)+\eps J^{(1)}(x_1,t),
\end{align}
where 
\begin{equation*}\label{eq:O1O2}
\Theta_1=\Lambda \coth(\Lambda)\qq\t{and}\qq \Theta_2=\frac{\Lambda}{\sh{\Lambda}},
\end{equation*}
and $K^{0},$ $K^{1}$ and $J^{1}$ are forcings that depend on $x_1$ and $t$. Such a model reflects the dynamics of the tumor colony up to an error proportional to
$$
\eps^2=\left(\frac{\text{width}}{\text{length}}\right)^2\ll 1.
$$
Besides the derivation of such a system, we also establish the following result.

\begin{theorem}
Let $(U_0,V_0)\in A^1$ be the zero mean initial data. Assume that
$$
\|U_0\|_{A^0}+\|V_0\|_{A^0}+\|U_0\|_{A^1}+\|V_0\|_{A^1}\leq C.
$$
Then, there exists a unique solution of \eqref{eq:patU}-\eqref{eq:patV} in the case $K^{(0)}=K^{(1)}=J^{(1)}=0$ such that
$$
(U,V)\in C([0,T];A^1\times A^1)\cap L^1(0,T;A^4\times A^4),
$$
for certain $0<T=T(\|U_0\|_{A^0}+\|V_0\|_{A^0}+\|U_0\|_{A^1}+\|V_0\|_{A^1})$.
\end{theorem}

\subsection{Notation}
We denote the one dimensional torus by $ \mathbb{S}^1 = \bR / 2\pi\bZ $. Alternatively, this domain can be thought as the interval $ [-\pi, \pi] $ with periodic boundary conditions. Given a matrix $ A \in \bR^{n\times m} $ we denote with $A_j^i $ the entry of $ A $ at row $ i $ and column $ j $, and we adopt Einstein convention for the summation of repeated indexes. We write
$$
f,_j=\frac{\partial f}{\partial x_j},\qquad \pat f=\frac{\partial f}{\partial t },
$$
for the space derivative in the $j-$th direction and for the time derivative, respectively. Let $v(x_1)$ denote a $L^2$ function on $\bT$ (as usually, identified with the interval $[-\pi,\pi]$ with periodic boundary conditions). We recall its Fourier series representation:
	\begin{equation*}
	\hat{v}(n) = \frac{1}{2\pi}\int _{\bT} v(x_1) e^{-in x_1}dx_1,
	\end{equation*}
where $ n\in\bZ $. Then we have that
\begin{equation*}
	v(x_1) =\sum_{n\in\bZ} \hat{v}(n) \ e^{in x_1}.
	\end{equation*}

\section{The dimensionless ALE formulation}
We perform the non-dimensionalization argument similar to the one in \citep{granero2025nonlocal} to deduce that the dimensionless free boundary problem is given by the system
\begin{subequations}\label{eq:pb-adim-LL-hb}
\begin{align}
\pat \beta-\eps D\D \b&=-N_1\b+M_1&&\text{ in }\Omega_T(t)\times(0,T),\label{eq:bT-adim-LL-hb}\\
\pat \sigma- \eps\D \si&=-N_2\sigma-N_3\b+M_2&&\text{ in }\Omega_T(t)\times(0,T),\label{eq:siT-adim-LL-hb}\\
-\D p+\theta\D \si&=\rho\pare{\si-\tau \b}+\omega&&\text{ in }\Omega_T(t)\times(0,T),\label{eq:divu-adim-LL-hb}\\
p&= -\eta\frac{h,_{11}}{(1+(\eps h,_1)^2)^{3/2}}&& \t{ on }\Gt(t)\times(0,T),\label{eq:tp-adim-LL-hb}\\
p&= \eta\frac{b,_{11}}{(1+(\eps b,_1)^2)^{3/2}}&& \t{ on }\Gb(t)\times(0,T),\label{eq:bp-adim-LL-hb}\\
\pat h&=\left(-\N p+\theta \N\si\right)\cdot (-\eps h,_1,1)&& \t{ on }\Gt(t)\times(0,T),\label{eq:dth-adim-LL-hb}\\
\pat b&=\left(-\N p+\theta \N\si\right)\cdot (-\eps b,_1,1)&& \t{ on }\Gb(t)\times(0,T),\label{eq:dtb-adim-LL-hb}\\
\si &=0&&\text{ on }\Gamma(t)\times(0,T),\label{eq:jumpsi-adim-LL-hb}\\
\b &= 0&&\text{ on }\Gamma(t)\times(0,T),\label{eq:jumpb-adim-LL-hb}
\end{align}
\end{subequations}
 where the dimensionless domain and boundaries are
\begin{align*}
\Omega _T\pare{t}&=\set{ x \in\mathbb{R}^2,\, x_1\in \mathbb{S}^1,\, -1+\eps b(x_1,t)< x_2<\eps h(x_1,t)},
\\
\Gt\pare{t}&=\set{ x \in\mathbb{R}^2,\, x_1\in \mathbb{S}^1,\, x_2=\eps h(x_1,t) },\\
\Gb\pare{t}&=\set{ x \in\mathbb{R}^2,\, x_1\in \mathbb{S}^1,\, x_2=-1+\eps b(x_1,t) },\\
\Ga\pare{t}&=\Gt(t)\cup\Gb(t),
\end{align*}
and the parameters involved are given by
\begin{equation*}\label{eq:eps-al}
\eps=\frac{H}{L},\q D=\frac{D_i}{D_n},
\end{equation*}
\begin{equation*}\label{eq:N123}
N_1=LH\frac{\de_i+\lm_i}{D_n},\q N_2=LH\frac{\de_n+\lm_n}{D_n},\q N_3=\frac{\ga_n  LH}{D_n},
\end{equation*}
\begin{equation*}\label{eq:M12}
M_1=\frac{LH}{D_n\tsi}\bra{-(\de_i+\lm_i)\b_D+\de_i\b_B}
,\q  M_2=\frac{LH}{D_n\tsi}
\bra{
-\si_D(\de_n+\lm_n)+\si_B\de_n-\ga_n\b_D
},
\end{equation*}
\begin{equation*}\label{eq:theta-rho-om-eta}
\theta=\frac{\chi\tsi}{D_n},\q \rho=\frac{\mu\tsi L^2}{D_n},\q \omega=\frac{\mu L^2}{D_n}\pare{\si_D-\tau\b_D-\tsi}
,\q \eta=\frac{\nu H}{L^2D_n},
\end{equation*}
being $L$ and $H$ typical length and width associated to the tumor colony. In particular, $H$ is the \emph{amplitude} or width of the tumor colony while $L$ is the \emph{length} of the colony. This means that, if the parameter $\varepsilon$ is small, the tumor is \emph{long} compared to the \emph{width}.

Our aim is to transform the free boundary problem \eqref{eq:pb-adim-LL-hb} into a fixed boundary one. With this goal in mind, we introduce the time dependent diffeomorphism
\[
 {\psi}(\cdot,t):\Omega\to\Omega(t)
\]
defined as
\[
 {\psi}(x_1,x_2,t)=\pare{x_1,x_2+(x_2+1)\eps h(x_1,t)-x_2\eps b(x_1,t)}.
\]
The reference domain and boundaries are
$$
\Omega_T=\set{ x \in\mathbb{R}^2, x_1\in \mathbb{S}^1,\;-1 <x_2<0},
$$
and
\begin{align*}
\Gt&=\set{ x \in\mathbb{R}^2, x_1\in \mathbb{S}^1,\;x_2=0},\\
\Gb&=\set{ x \in\mathbb{R}^2, x_1\in \mathbb{S}^1,\;x_2=-1},\\
\Ga&=\Gt\cup\Gb.
\end{align*}

We now compose the equations in problem \eqref{eq:pb-adim-LL-hb} with $ {\psi}$ to find its equivalent formulation over $\Omega_T$ in terms of
\[
S=\si \circ {\psi},\q  B=\b   \circ {\psi},\q  
 P=p\circ {\psi}.
\]
 Since we are interested in the asymptotic behavior in $\eps$, we rewrite  these equivalent equations collecting  all the terms which depend on $\eps$ in the right hand side. A simple computation (akin to the one in \citep{granero2025nonlocal}) shows that the problem in the fixed domain reads
\begin{align*}
\pat B+N_1B-M_1&=\eps\bra{D\D B+\pare{(x_2+1) \pat h-x_2\pat  b} B,_2}
&&\text{ in }\Omega_T\times(0,T), \\
\pat S+N_2S+N_3B-M_2&=\eps\bra{\D S+\pare{(x_2+1) \pat h-x_2\pat  b}S,_2}
&&\text{ in }\Omega_T\times(0,T), \\
-\D P+\theta\D S-\rho\pare{S-\tau B}-\omega&=-\eps Q(x,t)+\eps\theta Z(x,t)
&&\text{ in }\Omega_T\times(0,T), \\
P&= -\eta\frac{h,_{11}}{(1+(\eps h,_1)^2)^{3/2}}&& \t{ on }\Gt\times(0,T), \\
P&= \eta\frac{b,_{11}}{(1+(\eps b,_1)^2)^{3/2}}&& \t{ on }\Gb\times(0,T), \\
\pat h&=-P,_2+\theta S,_2+\eps h,_1\pare{P,_1-\theta S,_1}&&\nonumber\\
&\q+\eps\pare{h-b}\pare{P,_2-\theta S,_2}
&& \t{ on }\Gt\times(0,T), \\
\pat b&=-P,_2+\theta S,_2+\eps b,_1\pare{P,_1-\theta S,_1}&&\nonumber\\
&\q+\eps\pare{h-b}\pare{P,_2-\theta S,_2}
&& \t{ on }\Gb\times(0,T), \\
S &=0&&\text{ on }\Gamma\times(0,T), \\
B &= 0&&\text{ on }\Gamma\times(0,T), 
\end{align*}
where
\begin{align*}
Q(x,t)&=\pare{(x_2+1) h,_{11}-x_2 b,_{11}}P,_2  +2\pare{(x_2+1) h,_{1}-x_2 b,_{1}}P,_{12}+2 (h-b)P,_{22} , \\
Z(x,t)&=\pare{(x_2+1) h,_{11}-x_2 b,_{11}} S,_2 +2\pare{(x_2+1) h,_{1}-x_2 b,_{1}}S,_{12}+2 (h-b) S,_{22}. 
\end{align*}

\section{The linear regime}
As we are interested on capturing the main dynamics of tumor growth from a near flat (when compared to the length of the tumor) initial tumor colony, we introduce the ansatz
$$
h(x_1,t)=\sum_{n=0}^\infty\eps^n h^{(n)}(x_1,t),\quad b(x_1,t)=\sum_{n=0}^\infty\eps^n b^{(n)}(x_1,t),
$$
$$
B(x,t)=\sum_{n=0}^\infty\eps^n B^{(n)}(x,t),\q
S(x,t)=\sum_{n=0}^\infty\eps^n S^{(n)}(x,t),\quad P(x,t)=\sum_{n=0}^\infty\eps^n P^{(n)}(x,t),
$$
where
$$
B(x,0)=B^{(0)}(x,0),\q
S(x,0)=S^{(0)}(x,0),\q
h(x_1,0)=h^{(0)}(x_1,0),\q
b(x_1,0)=b^{(0)}(x_1,0),
$$
and
$$
S^{(n)}(x,0)=B^{(n)}(x,0)=0,\q h^{(n)}(x_1,0)=0,\q b^{(n)}(x_1,0)=0\q\t{for }n\ge1.
$$
We consider the case where
$$
M_1=M_2=\omega=O(\eps^2).
$$
This regime is meaningful in some real life situation such as, for instance, the case of avascular tumor growth with very little access to both nutrients and inhibitors or the case of vascular tumors where the amount of nutrients and inhibitors present in the system is very similar to the amount consumed by the tumor cells.

Furthermore, we also set
\[
D=1,\qq N_1=N_2=N_3=N.
\]

In what follows, we will continue the approach in \citep{granero2025nonlocal} (see also \citep{granero2025nonlocal} for the explanation of the method in a simple toy model).

The first terms in the asymptotic series are
\begin{subequations}\label{eq:pb-adim-LL-hb-fixed-phase0}
\begin{align}
\pat \mB{0}+N\mB{0}&=0&&\text{ in }\Omega_T\times(0,T),\label{eq:bT-adim-LL-hb-fixed-phase0}\\
\pat \mS{0}+N\mS{0}+N\mB{0}&=0
&&\text{ in }\Omega_T\times(0,T),\label{eq:siT-adim-LL-hb-fixed-phase0}\\
-\D \mP{0}+\theta\D \mS{0}-\rho\pare{\mS{0}-\tau \mB{0}}&=0&&\text{ in }\Omega_T\times(0,T),\label{eq:divu-adim-LL-hb-fixed-phase0}\\
\mP{0}&= -\eta \mh{0},_{11}&& \t{ on }\Gt\times(0,T),\label{eq:tp-adim-LL-hb-fixed-phase0}\\
\mP{0}&= \eta \mb{0},_{11}&& \t{ on }\Gb\times(0,T),\label{eq:bp-adim-LL-hb-fixed-phase0}\\
\pat \mh{0}&=-\mP{0},_2+\theta \mS{0},_2 
&& \t{ on }\Gt\times(0,T),\label{eq:dth-adim-LL-hb-fixed-phase0}\\
\pat \mb{0}&=-\mP{0},_2+\theta \mS{0},_2 && \t{ on }\Gb\times(0,T),\label{eq:dtb-adim-LL-hb-fixed-phase0}\\
\mS{0} &=0&&\text{ on }\Gamma\times(0,T),\label{eq:jumpsi-adim-LL-hb-fixed-phase0}\\
\mB{0} &=0&&\text{ on }\Gamma\times(0,T).\label{eq:jumpb-adim-LL-hb-fixed-phase0}
%
\end{align}
\end{subequations}

We want to rewrite the evolution equations for $h$ and $b$ in \eqref{eq:dth-adim-LL-hb-fixed-phase0} and \eqref{eq:dtb-adim-LL-hb-fixed-phase0} in terms of the initial data that we represent by
\begin{equation}\label{eq:data}
B(x,0)=-\cb x_2(x_2+1)\qq\t{and}\qq S(x,0)=-\cs x_2(x_2+1).
\end{equation}
These initial data mean that the quantity of nutrients and inhibitors only depend on the distance to the boundaries, \emph{i.e.} \emph{the deeper into the colony, the smaller the quantity of chemicals}. These initial data are chosen so the mathematical analysis can be carried out in a explicit way. If a different choice is made, some of the following computations can be carried out but they will not necessarily be explicit and some quantities will not necessarily be written in closed form.

We start with the computations for $B^{(0)}(x_2,t)$ and $S^{(0)}(x_2,t)$. The equations \eqref{eq:bT-adim-LL-hb-fixed-phase0} -- \eqref{eq:jumpb-adim-LL-hb-fixed-phase0} and \eqref{eq:siT-adim-LL-hb-fixed-phase0} -- \eqref{eq:jumpsi-adim-LL-hb-fixed-phase0} imply, respectively, that
\begin{align}
B^{(0)}(x_2,t)&=-\cb e^{-Nt}x_2(x_2+1),\label{eq:B0}\\
S^{(0)}(x_2,t)&=-  e^{-Nt}\b_1(t)x_2(x_2+1)
\label{eq:S0},
\end{align}
for
\begin{equation}
\b_1(t)=\cs-N \cb t.\label{eq:al1}
\end{equation}
Then, we deduce that
\[
S^{(0)},_2(x_2,t) =-  e^{-Nt}\b_1(t) (2x_2+1),
\]
from which
\begin{align}
\mS{0},_2(0,t)&=-  e^{-Nt}\b_1(t) ,\label{eq:S0-der2-top}\\
\mS{0},_2(-1,t)&=   e^{-Nt}\b_1(t) .\label{eq:S0-der2-bot}
\end{align}

We now focus on $P^{(0)}(x,t)$. We apply Lemma \ref{lem:sol-ell} to \eqref{eq:divu-adim-LL-hb-fixed-phase0}, \eqref{eq:tp-adim-LL-hb-fixed-phase0}, and \eqref{eq:bp-adim-LL-hb-fixed-phase0}, with $g_1=-\eta \mh{0},_{11}$, $g_2=\eta \mb{0},_{11}$, and $w=w^{(0)}$
\begin{align*}
w^{(0)}(x_2,t)&=\theta  \D S^{(0)}(x_2,t)-\rho \pare{S^{(0)}(x_2,t)-\tau B^{(0)}(x_2,t)} =-e^{-N t}\pare{2\theta\b_1(t)
-\rho\b_2(t)x_2(x_2+1) },\label{eq:w0-P0}
\end{align*}
being
\begin{equation}\label{eq:al2}
\b_2(t)=\b_1(t)-\tau \cb.
\end{equation}
Then, we obtain the following expression for $\mP{0}$:
\begin{align*}
P^{(0)}(x,t)
&=  x_2\int_{-1}^0(1+y_2)w(y_2)\,dy_2+\int_{0}^{x_2}(x_2-y_2)w(y_2)\,dy_2  \nonumber\\
&\q+\frac{\eta}{\sqrt{2\pi}}\sum_{k\in\mathbb{Z}}
\biggl\{
\frac{\sinh(|k|x_2)}{\sinh(|k|)} \av{k}^3 \int_{-1}^0{\sinh(|k|(1+y_2))}\pare{  \w{\mh{0}}(k,t)+y_2\pare{\w{\mh{0}}(k,t)+\w{\mb{0}}(k,t)}}\,dy_2 \nonumber \\
&\qq +\av{k}^3\int_{0}^{x_2}\sinh(|k|(x_2-y_2))\pare{\w{\mh{0}}(k,t)+y_2\pare{\w{\mh{0}}(k,t)+\w{\mb{0}}(k,t)}}\,dy_2\nonumber\\
& \qq+  k^2\pare{\w{\mh{0}}(k,t)+x_2\pare{\w{\mh{0}}(k,t)+\w{\mb{0}}(k,t)}}
\biggr\}e^{ikx_1}.  
\end{align*}
We simplify this formula by direct computations and find
\begin{align}
P^{(0)}(x,t)
&=  -\theta\b_1(t)e^{-N t}x_2(x_2+1)+\frac{\rho}{12}\b_2(t)e^{-N t}x_2(x_2+1)(x_2^2+x_2-1)  \nonumber\\
&\q+\frac{\eta}{\sqrt{2\pi}}\sum_{k\in\mathbb{Z}}
\biggl\{
k^2\pare{\cth{\av{k}}\sh{\av{k}x_2}+\ch{\av{k}x_2}-1-x_2}\w{\mh{0}}(k,t)
 \nonumber \\
&\qq +   k^2\pare{\frac{\sh{\av{k}x_2}}{\sh{\av{k}}}-x_2}\w{\mb{0}}(k,t)\nonumber\\
& \qq+  k^2\pare{\w{\mh{0}}(k,t)+x_2\pare{\w{\mh{0}}(k,t)+\w{\mb{0}}(k,t)}}
\biggr\}e^{ikx_1},\label{eq:P0}
\end{align}
being $\b_1(t)$ and $\b_2(t)$ as in \eqref{eq:al1} and \eqref{eq:al2}, respectively. Hence, its normal derivative is
\begin{align*}
&P^{(0)},_2(x,t)\nonumber\\
&=  -\theta\b_1(t)e^{-N t}(2x_2+1)+\frac{\rho}{12}\b_2(t)e^{-N t}(2x_2+1)(2x_2^2+2x_2-1)  \nonumber\\
&\q+\frac{\eta}{\sqrt{2\pi}}\sum_{k\in\mathbb{Z}}\av{k}^3
\biggl\{
\pare{\cth{\av{k}}\ch{\av{k}x_2}+\sh{\av{k}x_2}}\w{\mh{0}}(k,t)
 +    \frac{\ch{\av{k}x_2}}{\sh{\av{k}}}\w{\mb{0}}(k,t)
\biggr\}e^{ikx_1},  
\end{align*}
and we deduce that its values on the boundaries are
\begin{align}
P^{(0)},_2(x_1,0,t)
&=  -\theta\b_1(t)e^{-N t}-\frac{\rho}{12}\b_2(t)e^{-N t}  \nonumber\\
&\q+\frac{\eta}{\sqrt{2\pi}}\sum_{k\in\mathbb{Z}}\av{k}^3
\biggl\{
\cth{\av{k}}\w{\mh{0}}(k,t)
+   \frac{1}{\sh{\av{k}}}\w{\mb{0}}(k,t)
\biggr\}e^{ikx_1} \nonumber\\
&=
-\theta\b_1(t)e^{-N t}-\frac{\rho}{12}\b_2(t)e^{-N t} +\eta\pare{\Lambda\cth{\Lambda}\Lambda^2\mh{0}(x_1,t)+\frac{\Lambda}{\sh{\Lambda}}\Lambda^2\mb{0}(x_1,t)}
,
 \label{eq:P0-der2-top}
\end{align}
and
\begin{align}
P^{(0)},_2(x_1,-1,t)
&=  \theta\b_1(t)e^{-N t}+\frac{\rho}{12}\b_2(t)e^{-N t}  \nonumber\\
&\q+\frac{\eta}{\sqrt{2\pi}}\sum_{k\in\mathbb{Z}}\av{k}^3
\biggl\{
\pare{\cth{\av{k}}\ch{\av{k}}-\sh{\av{k}}}\w{\mh{0}}(k,t)
 +   \cth{\av{k}}\w{\mb{0}}(k,t)
\biggr\}e^{ikx_1}\nonumber\\
&=  \theta\b_1(t)e^{-N t}+\frac{\rho}{12}\b_2(t)e^{-N t}  \nonumber\\
&\q+\frac{\eta}{\sqrt{2\pi}}\sum_{k\in\mathbb{Z}}\av{k}^3
\biggl\{
\frac{1}{\sh{\av{k}}}\w{\mh{0}}(k,t)
 +   \cth{\av{k}}\w{\mb{0}}(k,t)
\biggr\}e^{ikx_1}\nonumber\\
&=\theta\b_1(t)e^{-N t}+\frac{\rho}{12}\b_2(t)e^{-N t}
+\eta\pare{
\frac{\Lambda}{\sh{\Lambda}}\Lambda^2\mh{0}(x_1,t)
+\Lambda\cth{\Lambda}\Lambda^2\mb{0}(x_1,t)
}
.  \label{eq:P0-der2-bot}
\end{align}

Finally, we focus on the expressions of $\pat h^{(0)}(x_1,t)$ and $\pat b^{(0)}(x_1,t)$ in \eqref{eq:dth-adim-LL-hb-fixed-phase0} and \eqref{eq:dtb-adim-LL-hb-fixed-phase0}. The previous computations imply that
\begin{align}
\pat h^{(0)}(x_1,t)&=-P^{(0)},_2(x_1,0,t)+\theta \mS{0},_2(0,t)
\nonumber\\
&=-\eta\pare{\Lambda\cth{\Lambda}\Lambda^2\mh{0}(x_1,t)+\frac{\Lambda}{\sh{\Lambda}}\Lambda^2\mb{0}(x_1,t)}+K^{(0)}(x_1,t),\label{eq:ht0}
\end{align}
and
\begin{align}
\pat\mb{0}(x_1,t)&=-\mP{0},_2(x_1,-1,t)+\theta \mS{0},_2(-1,t)\nonumber\\
&=-\eta\pare{
\frac{\Lambda}{\sh{\Lambda}}\Lambda^2\mh{0}(x_1,t)
+\Lambda\cth{\Lambda}\Lambda^2\mb{0}(x_1,t)
}-K^{(0)}(x_1,t),\label{eq:bt0}
\end{align}
for
\begin{align}
K^{(0)}(x_1,t)
&=\frac{\rho}{12}\b_2(t)e^{-Nt},
\label{eq:K0}
\end{align}
with $\b_2(t)$ as in \eqref{eq:al2} respectively.

\section{The weakly nonlinear correction}
We now consider the case $n=1$, i.e. the terms multiplied by $\eps$ in the asymptotic expansions. The simplification of parameters we did and the choice of initial data in \eqref{eq:data} implies that
\[
\D \mB{0}=\mB{0},_{22},\qq \D \mS{0}=\mS{0},_{22},\qq \mB{0},_1=\mS{0},_1=0.
\]
Hence, the system is given by
\begin{subequations}\label{eq:pb-adim-LL-hb-fixed-1}
\begin{align}
\pat \mB{1}+N\mB{1}&=\pare{(x_2+1) \pat \mh{0}-x_2\pat  \mb{0}} \mB{0},_2
+\mB{0},_{22}
&&\text{ in }\Omega_T\times(0,T),\label{eq:bT-adim-LL-hb-fixed-1}\\
\pat \mS{1}+N\mS{1}+N\mB{1}&= \pare{(x_2+1) \pat \mh{0}-x_2\pat  \mb{0}}\mS{0},_2
+\mS{0},_{22}
&&\text{ in }\Omega_T\times(0,T),\label{eq:siT-adim-LL-hb-fixed-1}\\
\D \mP{1}-\theta\D \mS{1}+\rho\pare{\mS{1}-\tau \mB{1}}&= Q^{(0)}(x,t)- \theta Z^{(0)}(x,t)
&&\text{ in }\Omega_T\times(0,T),\label{eq:divu-adim-LL-hb-fixed-1}\\
\mP{1}&= -\eta \mh{1},_{11}&& \t{ on }\Gt\times(0,T),\label{eq:tp-adim-LL-hb-fixed-1}\\
\mP{1}&= \eta\mb{1},_{11}&& \t{ on }\Gb\times(0,T),\label{eq:bp-adim-LL-hb-fixed-1}\\
\pat \mh{1}&=-\mP{1},_2+\theta \mS{1},_2+\mh{0},_1\mP{0},_1&&\nonumber\\
&\q+ \pare{\mh{0}-\mb{0}}\pare{\mP{0},_2-\theta \mS{0},_2}
&& \t{ on }\Gt\times(0,T),\label{eq:dth-adim-LL-hb-fixed-1}\\
\pat \mb{1}&=-\mP{1},_2+\theta \mS{1},_2+  \mb{0},_1\mP{0},_1&&\nonumber\\
&\q+ \pare{\mh{0}-\mb{0}}\pare{\mP{0},_2-\theta \mS{0},_2}
&& \t{ on }\Gb\times(0,T),\label{eq:dtb-adim-LL-hb-fixed-1}\\
\mS{1} &=0&&\text{ on }\Gamma\times(0,T),\label{eq:jumpsi-adim-LL-hb-fixed-1}\\
\mB{1} &= 0&&\text{ on }\Gamma\times(0,T).\label{eq:jumpb-adim-LL-hb-fixed-1}
\end{align}
\end{subequations}
where
\begin{align}
Z^{(0)}(x,t)&=\pare{(x_2+1) \mh{0},_{11}(x_1,t)-x_2 \mb{0},_{11}(x_1,t)} \mS{0},_2(x,t) \nonumber\\
&\q+2 (\mh{0}(x_1,t)-\mb{0}(x_1,t)) \mS{0},_{22}(x,t),\label{eq:Z0}\\
Q^{(0)}(x,t)&=\pare{(x_2+1) \mh{0},_{11}(x_1,t)-x_2 \mb{0},_{11}(x_1,t)}\mP{0},_2(x,t) \nonumber\\
&\q +2\pare{(x_2+1) \mh{0},_{1}(x_1,t)-x_2 \mb{0},_{1}(x_1,t)}\mP{0},_{12}(x,t)\nonumber\\
&\q+2 (\mh{0}(x_1,t)-\mb{0}(x_1,t))\mP{0},_{22}(x,t).\label{eq:Q0}
\end{align}
We can express $Z^{(0)}$ and $Q^{(0)}$ in terms of the initial data (see the definition of $\mS{0}$ in \eqref{eq:S0} and the definition of $\mP{0}$ in \eqref{eq:P0}).

Again, we want to rewrite $\pat\mh{1}$ and $\pat\mb{1}$ (see \eqref{eq:dth-adim-LL-hb-fixed-1} and \eqref{eq:dtb-adim-LL-hb-fixed-1}) in terms of themselves and the initial data.

We begin focusing on the terms
\[
 \pare{\mh{0}-\mb{0}}\pare{\mP{0},_2-\theta \mS{0},_2} .
\]
We recall
\[
 \mS{0},_2(0,t)\q\t{and}\q  \mS{0},_2(-1,t)\qq\t{in \eqref{eq:S0-der2-top} and  \eqref{eq:S0-der2-bot}}
\]
to deduce that
\begin{align}
-\theta \pare{\mh{0}(x_1,t)-\mb{0}(x_1,t)}\mS{0},_2(0,t)&=\theta \pare{\mh{0}(x_1,t)-\mb{0}(x_1,t)}\mS{0},_2(-1,t)\nonumber\\
&= \theta e^{-Nt}\b_1(t) \pare{\mh{0}(x_1,t)-\mb{0}(x_1,t)},\label{eq:S0-der2-P1-top-bot}
\end{align}
with $\b_1(t)$ as in \eqref{eq:al1}.\\
We now consider on the terms involving $P^{(0)},_1(x,t)$. Thanks to \eqref{eq:P0}, we   compute
\begin{align*}
P^{(0)},_1(x_1,0,t)
&=\frac{i\eta}{\sqrt{2\pi}}\sum_{k\in\mathbb{Z}}
k^3\w{\mh{0}}(k,t)e^{ikx_1},\\
P^{(0)},_1(x_1,-1,t)
&=-\frac{i\eta}{\sqrt{2\pi}}\sum_{k\in\mathbb{Z}}
k^3\w{\mb{0}}(k,t)e^{ikx_1},
\end{align*}
from which we deduce
\begin{align}
 \mh{0},_1(x_1,t)\mP{0},_1(x_1,0,t)&=\frac{1}{\sqrt{2\pi}}\sum_{k\in\ZZ}\reallywidehat{ \mh{0},_1\mP{0},_1}(k,0,t) e^{ikx_1}\nonumber
\\
&=-\frac{\eta}{\sqrt{2\pi}}\sum_{k,m\in\mathbb{Z}}m^3(k-m)\w{\mh{0}}(k-m,t)\w{\mh{0}}(m,t)e^{ikx_1},\label{eq:P0-der1-P1-top}\\
 \mb{0},_1(x_1,t)\mP{0},_1(x_1,-1,t)&=\frac{1}{\sqrt{2\pi}}\sum_{k\in\ZZ}\reallywidehat{ \mb{0},_1\mP{0},_1}(k,-1,t) e^{ikx_1}\nonumber
\\
&=\frac{\eta}{\sqrt{2\pi}}\sum_{k,m\in\mathbb{Z}}m^3(k-m)\w{\mb{0}}(k-m,t)\w{\mb{0}}(m,t)e^{ikx_1}.\label{eq:P0-der1-P1-bot}
\end{align}
We recall \eqref{eq:P0-der2-top} and \eqref{eq:P0-der2-bot} to compute the terms with the derivatives of $P^{(0)},_2(x,t)$:
\begin{align}
&\pare{\mh{0}(x_1,t)-\mb{0}(x_1,t)}\mP{0},_2(x_1,0,t)\nonumber\\
&=-\pare{\theta\b_1(t)e^{-N t}+\frac{\rho}{12}\b_2(t)e^{-N t}}\pare{\mh{0}(x_1,t)-\mb{0}(x_1,t)}\nonumber\\
&\q+\frac{\eta}{\sqrt{2\pi}}\sum_{k,\,m\in\mathbb{Z}}\av{m}^3
\biggl\{
\cth{\av{m}}\w{\mh{0}}(m,t)
+   \frac{1}{\sh{\av{m}}}\w{\mb{0}}(m,t)
\biggr\}\pare{\w{\mh{0}}(k-m,t)-\widehat{\mb{0}}(k-m,t)}e^{ikx_1} ,\label{eq:P0-der2-P1-top}
\end{align}
and
\begin{align}
&\pare{\mh{0}(x_1,t)-\mb{0}(x_1,t)}\mP{0},_2(x_1,-1,t)\nonumber\\
&=\pare{\theta\b_1(t)e^{-N t}+\frac{\rho}{12}\b_2(t)e^{-N t}}\pare{\mh{0}(x_1,t)-\mb{0}(x_1,t)}\nonumber\\
&\q+\frac{\eta}{\sqrt{2\pi}}\sum_{k,\,m\in\mathbb{Z}}\av{m}^3
\biggl\{
\frac{1}{\sh{\av{m}}}\w{\mh{0}}(m,t)
 +   \cth{\av{m}}\w{\mb{0}}(m,t)
\biggr\}\pare{\w{\mh{0}}(k-m,t)-\widehat{\mb{0}}(k-m,t)}e^{ikx_1}.\label{eq:P0-der2-P1-bot}
\end{align}

We gather the computations in  \eqref{eq:S0-der2-P1-top-bot}, \eqref{eq:P0-der1-P1-top} and \eqref{eq:P0-der2-P1-top}, to deduce that, so far,  the evolution equations of $\mh{1}$ in \eqref{eq:dth-adim-LL-hb-fixed-1} is given by
\begin{align}
\pat \mh{1}(x_1,t)&=
 -\frac{\eta}{\sqrt{2\pi}}\sum_{k,m\in\mathbb{Z}}m^3(k-m)\w{\mh{0}}(k-m,t)\w{\mh{0}}(m,t)e^{ikx_1}\nonumber\\
&\q +\frac{\eta}{\sqrt{2\pi}}\sum_{k,m\in\mathbb{Z}}|m|^3\coth(|m|)\pare{\w{\mh{0}}(k-m,t)-\w{\mb{0}}(k-m,t)}\w{\mh{0}}(m,t)e^{ikx_1}\nonumber\\
&\q+\frac{\eta}{\sqrt{2\pi}}\sum_{k,m\in\mathbb{Z}}\frac{|m|^3}{\sh{|m|}}\pare{\w{\mh{0}}(k-m,t)-\w{\mb{0}}(k-m,t)}\w{\mb{0}}(m,t)e^{ikx_1}\nonumber\\
&\q-\mP{1},_2(x_1,0,t)+\theta\mS{1}(x_1,0,t)-K_1^{(1)}(x_1,t),\label{eq:h1-preP1yS1}
\end{align}
while, from \eqref{eq:S0-der2-P1-top-bot}, \eqref{eq:P0-der1-P1-bot} and \eqref{eq:P0-der2-P1-bot}, we deduce that  the evolution equations of $\mb{1}$ in \eqref{eq:dtb-adim-LL-hb-fixed-1} is given by
\begin{align}
\pat \mb{1}(x_1,t)&=
\frac{\eta}{\sqrt{2\pi}}\sum_{k,m\in\mathbb{Z}}m^3(k-m)\w{\mb{0}}(k-m,t)\w{\mb{0}}(m,t)e^{ikx_1}\nonumber\\
&\q +\frac{\eta}{\sqrt{2\pi}} \sum_{k,m\in\mathbb{Z}}\frac{|m|^3}{\sinh(|m|)}\pare{\w{\mh{0}}(k-m,t)-\w{\mb{0}}(k-m,t)}\w{\mh{0}}(m,t)e^{ikx_1}\nonumber \\
&\q+\frac{\eta}{\sqrt{2\pi}}\sum_{k,m\in\mathbb{Z}}|m|^3\coth(|m|)\pare{\w{\mh{0}}(k-m,t)-\w{\mb{0}}(k-m,t)}\w{\mb{0}}(m,t)e^{ikx_1} \nonumber\\
&\q-\mP{1},_2(x_1,-1,t)+\theta\mS{1}(x_1,-1,t)+K_1^{(1)}(x_1,t),\label{eq:b1-preP1yS1}
\end{align}
with $K_1^{(1)}(x_1,t)$ given by
\begin{align*}
K_1^{(1)}(x_1,t)&=\frac{\rho}{12} e^{-Nt}\b_2(t) \pare{\mh{0}(x_1,t)-\mb{0}(x_1,t)}.
\end{align*}

We now turn our attention to $\mB{1}(x,t)$ and $\mS{1}(x,t)$. The ODE for $\mB{1}$ follows from \eqref{eq:bT-adim-LL-hb-fixed-1} -  \eqref{eq:jumpb-adim-LL-hb-fixed-1}. We use the explicit expression of $\mB{0}$ in \eqref{eq:B0} to rewrite  \eqref{eq:bT-adim-LL-hb-fixed-1} as
\begin{align}
\pat \mB{1}+N\mB{1}&=  \mB{0},_{22}+ \pare{(x_2+1)\pat\mh{0}-x_2\pat\mb{0}} \mB{0},_2&&\nonumber\\
& = -c_{\m{B}}e^{-Nt}\bra{ 2+\pare{(x_2+1)\pat\mh{0}-x_2\pat\mb{0}} (2x_2+1)}
&&\text{ in }\Omega_T\times(0,T).\label{eq:bT-adim-LL-hb-fixed-phase1}
\end{align}
We define
\begin{equation*}\label{eq:delta}
\diff{f(t)}=f(t)-f(0),
\end{equation*}
and we solve \eqref{eq:bT-adim-LL-hb-fixed-phase1}, obtaining
\begin{equation}
\mB{1}(x,t)=  -c_{\m{B}}e^{-Nt}\bra{ 2t+\pare{(x_2+1) \diff{\mh{0}(x_1,t)}-x_2 \diff{\mb{0}(x_1,t)}}(2x_2+1)}.\label{eq:B1}
\end{equation}

Similarly, we use  the expressions of $\mB{0},\, \mS{0}$ in  \eqref{eq:B0} - \eqref{eq:S0}   in the evolution equation of $\mS{1}$ \eqref{eq:siT-adim-LL-hb-fixed-1} to rewrite it as
\begin{align*}
&\pat \mS{1}+N\mS{1}\nonumber&&\\
&=-N\mB{1}+ \mS{0},_{22}+ \pare{(x_2+1)\pat\mh{0}-x_2\pat\mb{0}}\mS{0},_2 \nonumber&&\\
&=N  c_{\m{B}}e^{-Nt}\bra{ 2t+\pare{(x_2+1) \diff{\mh{0}(x_1,t)}-x_2 \diff{\mb{0}(x_1,t)}}(2x_2+1)}\nonumber&&\\
&\q-2e^{-Nt}\b_1(t) -e^{-Nt} \b_1(t)\pare{(x_2+1)\pat\mh{0}-x_2\pat\mb{0}} (2x_2+1)
&&\text{ in }\Omega_T\times(0,T),
\end{align*}
so
\begin{align}\label{eq:S1}
\mS{1}(x,t)
&=-e^{-Nt}\b_1(t)\pare{(2x_2+1)\pare{(x_2+1)\diff{\mh{0}(x_1,t)}-x_2\diff{\mb{0}(x_1,t)}}+2t}.
\end{align}

We can now compute the term $\mS{1},_2$ contained in both \eqref{eq:h1-preP1yS1} and  \eqref{eq:b1-preP1yS1}, obtaining that
\begin{align*}
\mS{1},_2(x,t)
&=-2e^{-Nt}\b_1(t)\pare{(x_2+1)\diff{\mh{0}(x_1,t)}-x_2\diff{\mb{0}(x_1,t)}}\\
&\q-e^{-Nt}\b_1(t)\pare{(2x_2+1)\pare{\diff{\mh{0}(x_1,t)}-\diff{\mb{0}(x_1,t)}}},
\end{align*}
from which their values on the boundaries are
\begin{align}
\mS{1},_2(x_1,0,t)
&=-3e^{-Nt}\b_1(t)\diff{\mh{0}(x_1,t)}+e^{-Nt}\b_1(t)\diff{\mb{0}(x_1,t)},\label{eq:S1-der2-top}\\
\mS{1},_2(x_1,-1,t)
&=-3e^{-Nt}\b_1(t)\diff{\mb{0}(x_1,t)}+e^{-Nt}\b_1(t)\diff{\mh{0}(x_1,t)}.\label{eq:S1-der2-bot}
\end{align}
The equalities \eqref{eq:S1-der2-top} and \eqref{eq:S1-der2-bot} can now be inserted into \eqref{eq:h1-preP1yS1} and \eqref{eq:b1-preP1yS1}.

Then, we are left with the term $\mP{1},_2$ to complete the above equations. We apply  Lemma \ref{lem:sol-ell2}  to  \eqref{eq:divu-adim-LL-hb-fixed-1}, \eqref{eq:tp-adim-LL-hb-fixed-1}, and \eqref{eq:bp-adim-LL-hb-fixed-1}, with $g_1=-\eta\mh{1},_{11}$, $g_2=\eta\mb{1},_{11}$, and $\w{w}=\w{w^{(1)}}$ given by
\begin{align*}
\w{w^{(1)}}(k,x_2,t)
&=\theta\w{\D \mS{1}}(k,x_2,t)-\rho\pare{\w{\mS{1}}(k,x_2,t)-\tau \w{\mB{1}}(k,x_2,t)}+\reallywidehat{\mQ{0}}(k,x_2,t)-\theta\reallywidehat{\mZ{0}}(k,x_2,t),
\end{align*}
so we  obtain that
 \begin{align*}
 \mP{1}(x,t)
&=\frac{1}{\sqrt{2\pi}}\sum_{k\in\mathbb{Z}}
\left\{
 \eta  |k|^2 \frac{\sinh(|k|(x_2+1))}{\sinh(|k|)}
\w{\mh{1}}(k,t)+\eta|k|^2\frac{\sinh(|k|x_2)}{\sinh(|k|)}\w{\mb{1}}(k,t)\right.\nonumber \\
&\left.\qq +\frac{\sinh(|k|x_2)}{\sinh(|k|)} \frac{1}{|k|}\int_{-1}^0\sinh(|k|(1+y_2))\w{w^{(1)}}(k,y_2,t)\,dy_2\right. \nonumber\\
&\left.\qq + \frac{1}{|k|}\int_{0}^{x_2}\sinh(|k|(x_2-y_2))\w{w^{(1)}}(k,y_2,t)\,dy_2
\right\}e^{ikx_1}.
\end{align*}
 The expression of $\w{w^{(1)}}$ in terms of the initial follows from the expressions of $\mS{0}$ \eqref{eq:S0}, of $\mP{0}$ \eqref{eq:P0}, of $\mS{1}$ \eqref{eq:S1}, of $\mB{1}$ \eqref{eq:B1}, of $\mQ{0}$ \eqref{eq:Q0}, and of $\mZ{0}$ \eqref{eq:Z0}.\\
In the next computations, we will need to evaluate in $x_2=0$ and $x_2=-1$ the following derivative
 \begin{align*}
\mP{1},_2(x,t)
&=\frac{1}{\sqrt{2\pi}}\sum_{k\in\mathbb{Z}}
\left\{
\eta  |k|^3 \frac{\cosh(|k|(x_2+1))}{\sinh(|k|)}
\w{\mh{1}}(k,t)+\eta|k|^3\frac{\cosh(|k|x_2)}{\sinh(|k|)}\w{\mb{1}}(k,t)\right.\nonumber \\
&\left.\qq +\frac{\cosh(|k|x_2)}{\sinh(|k|)} \int_{-1}^0\sinh(|k|(1+y_2))\w{w^{(1)}}(k,y_2,t)\,dy_2\right. \nonumber\\
&\left.\qq + \int_{0}^{x_2}\cosh(|k|(x_2-y_2))\w{w^{(1)}}(k,y_2,t)\,dy_2
\right\}e^{ikx_1}.
\end{align*}
After a lengthy computation (that we omit for the sake of readability) we find that 
\begin{align}
&\pat \mh{1}(x_1,t)\nonumber\\
&= -\eta  \pare{\Lambda \cth{\Lambda}\Lambda^2\mh{1}(x_1,t)
+\frac{\Lambda}{\sh{\Lambda}}\Lambda^2\mb{1}(x_1,t)}\nonumber \\
&\q+ \frac{\eta}{\sqrt{2\pi}} \sum_{k,m\in\mathbb{Z}}\w{\mh{0}}(k-m,t)\w{\mh{0}}(m,t)|m|^3\av{k} \pare{ \coth(|m|)\cth{\av{k}}-\frac{mk}{\av{k}|m|}}e^{ikx_1}\nonumber\\
&\q+\frac{\eta}{\sqrt{2\pi}}\sum_{k,m\in\mathbb{Z}}\frac{\cth{\av{k}}}{\sh{|m|}}|m|^3\av{k} \w{\mh{0}}(k-m,t) \w{\mb{0}}(m,t)e^{ikx_1}\nonumber\\
&\q -\frac{\eta}{\sqrt{2\pi}} \sum_{k,m\in\mathbb{Z}}\frac{|m|^3\av{k}}{\sh{\av{k}}\sh{\av{m}}}   \w{\mb{0}}(k-m,t) \w{\mh{0}}(m,t)e^{ikx_1}\nonumber\\
&\q-\frac{\eta}{\sqrt{2\pi}}\sum_{k,m\in\mathbb{Z}}\frac{\cth{\av{m}}}{\sh{|k|}}|m|^3\av{k}\w{\mb{0}}(k-m,t)\w{\mb{0}}(m,t)e^{ikx_1}\nonumber\\
&\q-\frac{1}{\sqrt{2\pi}}\sum_{k\in\ZZ}\bra{
\theta\b_1(t)e^{-Nt}
+\frac{\rho}{12}\b_2(t)e^{-Nt}
}\av{k}\cth{\av{k}}\w{\mh{0}}(k,t)e^{ikx_1}\nonumber\\
&\q-\frac{1}{\sqrt{2\pi}}\sum_{k\in\ZZ}\bra{
\theta\b_1(t)e^{-Nt}+
\frac{\rho}{12}\b_2(t)e^{-Nt}
}\frac{\av{k}}{\sh{\av{k}}}\w{\mb{0}}(k,t)e^{ikx_1}\nonumber\\
&\q+\frac{1}{\sqrt{2\pi}}\sum_{k\in\ZZ}\bra{
\theta e^{-Nt}\b_1(t)\av{k}\cth{\av{k}}
+\rho e^{-Nt}\b_2(t)\pare{\frac{\av{k}\cth{\av{k}}-3}{k^2}+4\frac{\ch{\av{k}}-1}{\av{k}^3\sh{\av{k}}}}
}\w{\mh{0}}(k,0)e^{ikx_1}\nonumber\\
&\q-\frac{1}{\sqrt{2\pi}}\sum_{k\in\ZZ}\bra{
-\theta e^{-Nt}\b_1(t)\frac{\av{k}}{\sh{\av{k}}}
+\rho e^{-Nt}\b_2(t)\pare{-\frac{\sh{\av{k}}+\av{k}}{k^2\sh{\av{k}}}+4\frac{\ch{\av{k}}-1}{\av{k}^3\sh{\av{k}}}}
}\w{\mb{0}}(k,0)e^{ikx_1}\nonumber\\
&\q-2 \rho e^{-Nt}\b_2(t)t\frac{1}{\sqrt{2\pi}}\sum_{k\in\ZZ}\frac{\ch{\av{k}}-1}{\av{k}\sh{\av{k}}}e^{ikx_1},\label{eq:h1}
\end{align}
and
\begin{align}
&\pat \mb{1}(x_1,t)\nonumber\\
&=-\eta  \bra{\frac{\Lambda}{\sh{\Lambda}}\Lambda^2\mh{1}(x_1,t)
+\Lambda\cth{\Lambda}\Lambda^2\mb{1}(x_1,t)}\nonumber \\
&\q+\frac{\eta}{\sqrt{2\pi}} \sum_{k,m\in\mathbb{Z}}\frac{\cth{|m|}}{\sh{\av{k}}}\av{k}\av{m}^3\w{\mh{0}}(k-m,t)\w{\mh{0}}(m,t)e^{ikx_1}\nonumber\\
&\q+\frac{\eta}{\sqrt{2\pi}}\sum_{k,m\in\mathbb{Z}}\frac{|k||m|^3}{\sh{\av{k}}\sh{|m|}}\w{\mh{0}}(k-m,t)\w{\mb{0}}(m,t)e^{ikx_1} \nonumber\\
&\q -\frac{\eta}{\sqrt{2\pi}} \sum_{k,m\in\mathbb{Z}}\frac{\cth{\av{k}}}{\sh{|m|}}\av{k}\av{m}^3 \w{\mb{0}}(k-m,t)\w{\mh{0}}(m,t)e^{ikx_1}\nonumber \\
&\q+\frac{\eta}{\sqrt{2\pi}}\sum_{k,m\in\mathbb{Z}}\pare{\frac{mk}{\av{m}\av{k}}-\cth{\av{m}}\cth{\av{k} }}\av{m}^3\av{k}\w{\mb{0}}(k-m,t)\w{\mb{0}}(m,t)e^{ikx_1}\nonumber\\
&\q
-\frac{1}{\sqrt{2\pi}}\sum_{k\in\ZZ}\bra{
\theta e^{-Nt}\b_1(t)
+
\frac{\rho}{12} e^{-Nt}\b_2(t)
}\frac{\av{k}}{\sh{\av{k}}}
\w{\mh{0}}(k,t)e^{ikx_1}\nonumber\\
&\q-\frac{1}{\sqrt{2\pi}}\sum_{k\in\ZZ}\bra{
\theta e^{-Nt}\b_1(t)
+
\frac{\rho}{12} e^{-Nt}\b_2(t)
}\av{k}\cth{\av{k}}
\w{\mb{0}}(k,t)e^{ikx_1}\nonumber\\
&\q -\frac{1}{\sqrt{2\pi}}\sum_{k\in\ZZ}\bra{
-\theta e^{-Nt}\b_1(t)
\frac{\av{k}}{\sh{\av{k}}}
+
\rho e^{-Nt}\b_2(t)\pare{4\frac{\cth{\av{k}}-1}{\av{k}^3\sh{\av{k}}}-\frac{1}{k^2}-\frac{1}{\av{k}\sh{\av{k}}}}}
\w{\mh{0}}(k,0)e^{ikx_1}\nonumber\\
&\q-\frac{1}{\sqrt{2\pi}}\sum_{k\in\ZZ}\bra{
-\theta e^{-Nt}\b_1(t)\av{k}\cth{\av{k}}
+
\rho e^{-Nt}\b_2(t)\pare{\frac{4-\ch{\av{k}}}{\av{k}^3\sh{\av{k}}}+\frac{3}{k^2}-\frac{\cth{\av{k}}}{\av{k}}}
}
\w{\mb{0}}(k,0)e^{ikx_1}\nonumber\\
&\q -2\rho e^{-Nt}\b_2(t)t\frac{1}{\sqrt{2\pi}}\sum_{k\in\ZZ}\frac{1-\ch{\av{k}}}{\av{k}\sh{\av{k}}}e^{ikx_1} .\label{eq:b1}
\end{align}

\section{The asymptotic model of a multilayer tumor with two free boundaries}
We define the operators
\begin{equation*}\label{eq:O1O22}
\Theta_1=\Lambda \coth(\Lambda)\qq\t{and}\qq \Theta_2=\frac{\Lambda}{\sh{\Lambda}},
\end{equation*}
so we rewrite asymptotic equations of $\pat\mh{0}$ and $\pat\mb{0}$ in \eqref{eq:ht0} and \eqref{eq:bt0} as
\begin{align*}
\pat h^{(0)}(x_1,t)
&=-\eta\pare{\Theta_1\Lambda^2\mh{0}(x_1,t)+\Theta_2\Lambda^2\mb{0}(x_1,t)}+K^{(0)}(x_1,t),
\end{align*}
and
\begin{align*}
\pat\mb{0}(x_1,t)
&=-\eta\pare{
\Theta_2\Lambda^2\mh{0}(x_1,t)
+\Theta_1\Lambda^2\mb{0}(x_1,t)
}-K^{(0)}(x_1,t),
\end{align*}
with $K^{(0)}$ as in \eqref{eq:K0}, i.e.
\[
K^{(0)}(x_1,t)
=\frac{\rho}{12}\b_2(t)e^{-Nt};
\]
and the asymptotic equations of $\pat\mh{1}$ and $\pat\mb{1}$ in \eqref{eq:h1} and \eqref{eq:b1} as
\begin{align*}
\pat\mh{1}&=-\eta \Lambda(\sinh(\Lambda))^{-1}\bra{\cosh(\Lambda)\Lambda^2\mh{1}(x_1,t)+\Lambda^2\mb{1}(x_1,t)}\nonumber\\
&\q+\eta
\Theta_1\pare{\mh{0}(x_1,t)\Theta_1\Lambda^2\mh{0}(x_1,t)}
+\eta
\Theta_1\pare{\mh{0}(x_1,t)\Theta_2\Lambda^2\mb{0}(x_1,t)}\nonumber\\
&\q+\eta\pare{\mh{0}(x_1,t)\Lambda^2\mh{0}(x_1,t),_1},_1\nonumber\\
&\q -\eta\Theta_2\pare{\mb{0}(x_1,t)\Theta_2\Lambda^2\mh{0}(x_1,t)}-\eta
\Theta_2\pare{\mb{0}(x_1,t)\Theta_1\Lambda^2\mb{0}(x_1,t)}\nonumber
\\
&\q -\bra{
\theta\b_1(t)e^{-Nt}+
\frac{\rho}{12}\b_2(t)e^{-Nt}
}
\pare{\Theta_1 \mh{0}(x_1,t)+\Theta_2 \mb{0}(x_1,t)}
\nonumber\\
&\q +K^{(1)}(x_1,t),
\end{align*}
and
\begin{align*}
\pat \mb{1}
&=-\eta\pare{\frac{\Lambda}{\sh{\Lambda}}\Lambda^2\mh{1}(x_1,t)+\Lambda\coth(\Lambda)\Lambda^2\mb{1}(x_1,t)}\nonumber\\
&\q+ \eta\Theta_2\pare{\mh{0}(x_1,t)\Theta_1\Lambda^2\mh{0}(x_1,t)}+\eta\Theta_2\pare{\mh{0}(x_1,t)\Theta_2\Lambda^2\mb{0}(x_1,t)}\nonumber\\
&\q-\eta\Theta_1\pare{\mb{0}(x_1,t)\Theta_2\Lambda^2\mh{0}(x_1,t)}-\eta\Theta_1\pare{\mb{0}(x_1,t)\Theta_1\Lambda^2\mb{0}(x_1,t)} \nonumber\\
&\q -\eta \pare{\mb{0}(x_1,t)\Lambda^2\mb{0},_1(x_1,t)},_1\nonumber\\
&\q -\bra{
\theta e^{-Nt}\b_1(t)
+
\frac{\rho}{12} e^{-Nt}\b_2(t)
}\pare{\Theta_2 \mh{0}(x_1,t)+\Theta_1 \mb{0}(x_1,t)}\nonumber\\
&\q +J^{(1)}(x_1,t),
\end{align*}
with
\begin{align*}
&K^{(1)}(x_1,t)\nonumber\\
&=\frac{1}{\sqrt{2\pi}}\sum_{k\in\ZZ}\bra{
\theta e^{-Nt}\b_1(t)\av{k}\cth{\av{k}}
+\rho e^{-Nt}\b_2(t)\pare{\frac{\av{k}\cth{\av{k}}-3}{k^2}+4\frac{\ch{\av{k}}-1}{\av{k}^3\sh{\av{k}}}}
}\w{\mh{0}}(k,0)e^{ikx_1}\nonumber\\
&\q-\frac{1}{\sqrt{2\pi}}\sum_{k\in\ZZ}\bra{
-\theta e^{-Nt}\b_1(t)\frac{\av{k}}{\sh{\av{k}}}
+\rho e^{-Nt}\b_2(t)\pare{-\frac{\sh{\av{k}}+\av{k}}{k^2\sh{\av{k}}}+4\frac{\ch{\av{k}}-1}{\av{k}^3\sh{\av{k}}}}
}\w{\mb{0}}(k,0)e^{ikx_1}\nonumber\\
&\q-2 \rho e^{-Nt}\b_2(t)t\frac{1}{\sqrt{2\pi}}\sum_{k\in\ZZ}\frac{\ch{\av{k}}-1}{\av{k}\sh{\av{k}}}e^{ikx_1},
\\
&J^{(1)}(x_1,t)\nonumber\\
&=-\frac{1}{\sqrt{2\pi}}\sum_{k\in\ZZ}\bra{
-\theta e^{-Nt}\b_1(t)
\frac{\av{k}}{\sh{\av{k}}}
+
\rho e^{-Nt}\b_2(t)\pare{4\frac{\cth{\av{k}}-1}{\av{k}^3\sh{\av{k}}}-\frac{1}{k^2}-\frac{1}{\av{k}\sh{\av{k}}}}}
\w{\mh{0}}(k,0)e^{ikx_1}\nonumber\\
&\q-\frac{1}{\sqrt{2\pi}}\sum_{k\in\ZZ}\bra{
-\theta e^{-Nt}\b_1(t)\av{k}\cth{\av{k}}
+
\rho e^{-Nt}\b_2(t)\pare{\frac{4-\ch{\av{k}}}{\av{k}^3\sh{\av{k}}}+\frac{3}{k^2}-\frac{\cth{\av{k}}}{\av{k}}}
}
\w{\mb{0}}(k,0)e^{ikx_1}\nonumber\\
&\q -2\rho e^{-Nt}\b_2(t)t\frac{1}{\sqrt{2\pi}}\sum_{k\in\ZZ}\frac{1-\ch{\av{k}}}{\av{k}\sh{\av{k}}}e^{ikx_1} .
\end{align*}

We set
\[
U(x_1,t)=\mh{0}(x_1,t)+\eps\mh{1}(x_1,t)
\qq\t{and}\qq
V(x_1,t)=\mb{0}(x_1,t)+\eps\mb{1}(x_1,t).
\]
Then, the $U$ and $V$ satisfy the following system of equations up to a $O(\varepsilon^2)$ correction term
\begin{align}\label{eq:patU2}
\pat U(x_1,t)
&=-\eta\bra{\Theta_1 \pare{\Lambda^2 U(x_1,t)}+\Theta_2 \pare{\Lambda^2 V(x_1,t) }}\nonumber\\
&\q+\eps\eta
\Theta_1\pare{U(x_1,t)\Theta_1\Lambda^2U(x_1,t)}
+\eps\eta
\Theta_1\pare{U(x_1,t)\Theta_2\Lambda^2V(x_1,t)}\nonumber\\
&\q+\eps\eta\pare{U(x_1,t)\Lambda^2U(x_1,t),_1},_1\nonumber\\
&\q -\eps\eta\Theta_2\pare{V(x_1,t)\Theta_2\Lambda^2U(x_1,t)}-\eps\eta
\Theta_2\pare{V(x_1,t)\Theta_1\Lambda^2V(x_1,t)}\nonumber
\\
&\q -\eps\bra{
\theta\b_1(t)e^{-Nt}+
\frac{\rho}{12}\b_2(t)e^{-Nt}
}
\pare{\Theta_1 U(x_1,t)+\Theta_2 V(x_1,t)}\nonumber\\
&\q +K^{(0)}(x_1,t)+\eps  K^{(1)}(x_1,t).
\end{align}
and
\begin{align}\label{eq:patV2}
\pat V(x_1,t)
&=-\eta\bra{
\Theta_2\pare{\Lambda^2U(x_1,t)}+\Theta_1\pare{\Lambda^2V(x_1,t)}
}\nonumber\\
&\q+ \eps\eta\Theta_2\pare{U(x_1,t)\Theta_1\Lambda^2U(x_1,t)}+\eps\eta\Theta_2\pare{U(x_1,t)\Theta_2\Lambda^2V(x_1,t)}\nonumber\\
&\q-\eps\eta\Theta_1\pare{V(x_1,t)\Theta_2\Lambda^2U(x_1,t)}-\eps\eta\Theta_1\pare{V(x_1,t)\Theta_1\Lambda^2V(x_1,t)} \nonumber\\
&\q -\eps\eta \pare{V(x_1,t)\Lambda^2V,_1(x_1,t)},_1\nonumber\\
&\q -\eps\bra{
\theta e^{-Nt}\b_1(t)
+
\frac{\rho}{12} e^{-Nt}\b_2(t)
}\pare{\Theta_2 U(x_1,t)+\Theta_1 V(x_1,t)}\nonumber\\
&\q -K^{(0)}(x_1,t)+\eps J^{(1)}(x_1,t).
\end{align}

\section{An existence result}
In this section we consider the following simplified system
\begin{align}\label{eq:patU_Sim}
\pat U(x_1,t)
&=-\eta\bra{\Theta_1 \pare{\Lambda^2 U(x_1,t)}+\Theta_2 \pare{\Lambda^2 V(x_1,t) }}\nonumber\\
&\q+\eps\eta
\Theta_1\pare{U(x_1,t)\Theta_1\Lambda^2U(x_1,t)}
+\eps\eta
\Theta_1\pare{U(x_1,t)\Theta_2\Lambda^2V(x_1,t)}\nonumber\\
&\q+\eps\eta\pare{U(x_1,t)\Lambda^2U(x_1,t),_1},_1\nonumber\\
&\q -\eps\eta\Theta_2\pare{V(x_1,t)\Theta_2\Lambda^2U(x_1,t)}-\eps\eta
\Theta_2\pare{V(x_1,t)\Theta_1\Lambda^2V(x_1,t)}\nonumber
\\
&\q -\eps\bra{
\theta\b_1(t)e^{-Nt}+
\frac{\rho}{12}\b_2(t)e^{-Nt}
}
\pare{\Theta_1 U(x_1,t)+\Theta_2 V(x_1,t)}.
\end{align}
and
\begin{align}\label{eq:patV_Sim}
\pat V(x_1,t)
&=-\eta\bra{
\Theta_2\pare{\Lambda^2U(x_1,t)}+\Theta_1\pare{\Lambda^2V(x_1,t)}
}\nonumber\\
&\q+ \eps\eta\Theta_2\pare{U(x_1,t)\Theta_1\Lambda^2U(x_1,t)}+\eps\eta\Theta_2\pare{U(x_1,t)\Theta_2\Lambda^2V(x_1,t)}\nonumber\\
&\q-\eps\eta\Theta_1\pare{V(x_1,t)\Theta_2\Lambda^2U(x_1,t)}-\eps\eta\Theta_1\pare{V(x_1,t)\Theta_1\Lambda^2V(x_1,t)} \nonumber\\
&\q -\eps\eta \pare{V(x_1,t)\Lambda^2V,_1(x_1,t)},_1\nonumber\\
&\q -\eps\bra{
\theta e^{-Nt}\b_1(t)
+
\frac{\rho}{12} e^{-Nt}\b_2(t)
}\pare{\Theta_2 U(x_1,t)+\Theta_1 V(x_1,t)},
\end{align}
{\it i.e.} we set $K^{(0)}=K^{(1)}=J^{(1)}=0$ in \eqref{eq:patU2} and \eqref{eq:patV2}.

For a fixed integer $j$, Let us introduce the Wiener space
$$
A^j=\set{u:\q \w{\partial_x^j u}\in \ell^1,\q j\ge0}
$$
with norm
$$
\|u\|_{A^j}=\norm{\w{\partial_x^j u}}_{\ell^1}
$$
These spaces form a Banach algebra. Equipped with these spaces we can prove the following result:
\begin{theorem}
Let $(U_0,V_0)\in A^1$ be the zero mean initial data. Assume that
$$
\|U_0\|_{A^0}+\|V_0\|_{A^0}+\|U_0\|_{A^1}+\|V_0\|_{A^1}\leq C.
$$
Then, there exists a unique solution of \eqref{eq:patU_Sim}-\eqref{eq:patV_Sim}
$$
(U,V)\in C([0,T];A^1\times A^1)\cap L^1(0,T;A^4\times A^4),
$$
for certain $0<T=T(\|U_0\|_{A^0}+\|V_0\|_{A^0}+\|U_0\|_{A^1}+\|V_0\|_{A^1})$.
\end{theorem}
\begin{proof}
We define the energies
$$
E^j(t)=\|U(t)\|_{A^j}+\|V(t)\|_{A^j},\q j\ge0.
$$
As the existence of approximate solutions can be done as in our companion paper \citep{granero2025nonlocal}, we focus only on the energy estimates. Without loss of generality we fix $\eps=\eta=1$.

We start observing that there are a number of regularizing operators in the system of PDEs. Indeed
$$
\coth(|k|)=1+\frac{2}{e^{2|k|}-1} \qq\text{and}\qq\Theta_2\Lambda^n\leq C_n.
$$
We start estimating the $E^0$ energy. We have the following terms corresponding to the equation for $U$:
\begin{align*}
\frac{d}{dt} \|U(t)\|_{A^0}
&=-\|\Theta_1 \pare{\Lambda^2 U(t)}\|_{A^0}+c\|V(t)\|_{A^0}\nonumber\\
&\q+\|
\Theta_1\pare{U(t)\Theta_1\Lambda^2U(t)}+\pare{U(t)\Lambda^2U(t),_1},_1\|_{A^0}\nonumber\\
&\q+
\|\Theta_1\pare{U(t)\Theta_2\Lambda^2V(t)}\|_{A^0}\nonumber\\
&\q +\|\Theta_2\pare{V(t)\Theta_2\Lambda^2U(t)}\|_{A^0}+\|
\Theta_2\pare{V(t)\Theta_1\Lambda^2V(t)}\|_{A^0}.
\end{align*}
Using the previous identity for $\coth(\cdot)$, we find that the term
$$
NL_1(t)=\|\Theta_1\pare{U(t)\Theta_1\Lambda^2U(t)}+\pare{U(t)\Lambda^2U(t),_1},_1\|_{A^0}
$$
has a commutator structure that can be exploited. Indeed, if we pass to Fourier variables, we find the following
$$
NL_1(t)=\sum_{k,m\in\ZZ}\w{U}(k-m,t)\w{U}(m,t)\left(\coth(|k|)|k|\coth(|k-m|)|k-m|^3-k(k-m)^3\right).
$$
Thus, it can be written as
$$
NL_1(t)\leq\sum_{k,m\in\ZZ}\w{U}(k-m,t)\w{U}(m,t)\left(|k||k-m|^3-k(k-m)^3\right) +c(\|U(t)\|_{A^0}+\|U(t)\|_{A^1})\|U(t)\|_{A^3}.
$$
The multiplier
$$
 |k||k-m|^3-k(k-m)^3 
$$
vanishes unless $\text{sgn}(k)\neq\text{sgn}(k-m)$. In this case we have that either
$$
k<0<k-m\qq\text{or}\qq k-m<0<k.
$$
In both cases that implies
$$
|k|\leq |m|.
$$
As a consequence,
$$
NL_1(t)\leq c(\|U(t)\|_{A^0}+\|U(t)\|_{A^1})\|U(t)\|_{A^3}\le c\pare{E^0(t)+E^1(t)}\|U(t)\|_{A^3}.
$$
Similarly, we have that
$$
NL_2(t)=\|\Theta_1\pare{U(t)\Theta_2\Lambda^2V(t)}\|_{A^0}\leq c\|U(t)\|_{A^1}\|V(t)\|_{A^0}\le cE^0(t)E^1(t),
$$
$$
NL_3(t)=\|\Theta_2\pare{V(t)\Theta_2\Lambda^2U(t)}\|_{A^0}\leq c\|U(t)\|_{A^0}\|V(t)\|_{A^0}\le cE^0(t)^2,
$$
$$
NL_4(t)=\|\Theta_2\pare{V(t)\Theta_1\Lambda^2V(t)}\|_{A^0}\leq c\|V(t)\|_{A^0}\|V(t)\|_{A^3}\le cE^0(t)\|V(t)\|_{A^3}.
$$
Collecting every estimate we find that
$$
\frac{d}{dt}\|U(t)\|_{A^0}\leq c\pare{E^0(t)+E^1(t)}E^3(t)+c\pare{1+E^0(t)+E^1(t)}^2-\|U(t)\|_{A^3}.
$$
Following the same procedure, we find that
$$
\frac{d}{dt}\|V(t)\|_{A^0}\leq c\pare{E^0(t)+E^1(t)}E^3(t)+c\pare{1+E^0(t)+E^1(t)}^2-\|V(t)\|_{A^3}.
$$
Thus,
$$
\frac{d}{dt}E^0(t)\leq c\pare{E^0(t)+E^1(t)}E^3(t)+c\pare{1+E^0(t)+E^1(t)}^2-E^3(t).
$$

We now continue with the $E^1$ energy. The higher order term corresponds to
$$
NL_5(t)=\|\Theta_1\pare{U(t)\Theta_1\Lambda^2U(t)}+\pare{U(t)\Lambda^2U(t),_1},_1\|_{A^1},
$$
which in Fourier variables can be written as
$$
NL_5(t)=\sum_{k,m\in\ZZ}\w{U}(k-m,t)\w{U}(m,t)|k|\left(\coth(|k|)|k|\coth(|k-m|)|k-m|^3-k(k-m)^3\right).
$$
We observe that the same commutator structure appears and the bound
$$
NL_5(t)\leq c(\|U(t)\|_{A^0}+\|U(t)\|_{A^2})\|U(t)\|_{A^3}
$$
follows. Interpolating, we find that
$$
NL_5(t)\leq c(\|U(t)\|_{A^0}+\|U(t)\|_{A^1})\|U(t)\|_{A^4}.
$$
Following the previous ideas together with interpolation, we find the estimate
$$
\frac{d}{dt}E^1(t)\leq c\pare{E^0(t)+E^1(t)}E^4(t)+c\pare{1+E^0(t)+E^1(t)}^2-E^4(t).
$$
Collecting both estimates for the energies we conclude that, for small enough initial data,
$$
\frac{d}{dt}\pare{E^0(t)+E^1(t)}\leq c\pare{1+E^0(t)+E^1(t)}^2,
$$
which guarantees a uniform time of existence. The uniqueness follows from a standard contradiction argument as in \citep{granero2025nonlocal}.
\end{proof}

\appendix
\section{The solution of the Poisson equation}

\begin{lemma}\label{lem:sol-ell}
Consider the problem
\begin{align*}
\D u(x_1,x_2)&= w(x_2)&&\t{in }\mathbb{S}^1\times(-1,0),\\
u(x_1,0)&=g_1(x_1)&&\t{for }x_1\in \mathbb{S}^1,\\
u(x_1,-1)&=g_2(x_1)&&\t{for }x_1\in \mathbb{S}^1.
\end{align*}
Then
\begin{align}\label{eq:u}
u(x)&=
  x_2\int_{-1}^0(1+y_2)w(y_2)\,dy_2+\int_{0}^{x_2}(x_2-y_2)w(y_2)\,dy_2  \nonumber\\
&\q+\frac{1}{\sqrt{2\pi}}\sum_{k\in\mathbb{Z}}
\biggl\{
\frac{\sinh(|k|x_2)}{\sinh(|k|)} \av{k}\int_{-1}^0{\sinh(|k|(1+y_2))}\pare{\w{g_1}(k)+y_2\pare{\w{g_1}(k)-\w{g_2}(k)}}\,dy_2 \nonumber \\
&\qq + \av{k}\int_{0}^{x_2}\sinh(|k|(x_2-y_2))\pare{\w{g_1}(k)+y_2\pare{\w{g_1}(k)-\w{g_2}(k)}}\,dy_2\nonumber\\
& \qq+\w{g_1}(k)+x_2(\w{g_1}(k)-\w{g_2}(k))
\biggr\}e^{ikx_1}.
\end{align}
Furthermore,
\begin{equation*}
u,_1(x_1,0)=g_{1,1}(x_1),\qq
u,_1(x_1,-1)=g_{2,1}(x_1),
\end{equation*}
and
\begin{align*}
u,_2(x_1,0)&= (\sh{\Lambda})^{-1}\int_{-1}^{0}\sh{\Lambda(1+y_2)}\pare{w(y_2) +\Lambda^2\pare{g_1(x_1)+y_2\pare{g_1(x_1)-g_2(x_1)}}}\,dy_2\\
&\q +  g_1(x_1)-g_2(x_1),
\\
u,_2(x_1,-1)
&=\int_{-1}^{0}\pare{\cth{\Lambda}\sh{\Lambda(1+y_2)}-\ch{\Lambda(1+y_2)}}\pare{w(y_2) +\Lambda^2\pare{g_1(x_1)+y_2\pare{g_1(x_1)-g_2(x_1)}}}\,dy_2\\
&\q+ g_1(x_1)-g_2(x_1).
\end{align*}
\end{lemma}
\begin{proof}

We set
\begin{equation}\label{eq:v}
v(x)=u(x)-\pare{g_1(x_1)+x_2(g_1(x_1)-g_2(x_1))},
\end{equation}
which satisfy the problem
\begin{align*}
\D v(x)&= w(x_2)-\pare{g_{1,11}(x_1)+x_2(g_{1,11}(x_1)-g_{2,11}(x_1))}&&\t{in }\mathbb{S}^1\times(-1,0),\\
v(x_1,0)&=0&&\t{for }x_1\in \mathbb{S}^1,\\
v(x_1,-1)&=0&&\t{for }x_1\in \mathbb{S}^1.
\end{align*}
Let
\begin{equation}\label{eq:f}
f(x)=w(x_2)-\pare{g_{1,11}(x_1)+x_2(g_{1,11}(x_1)-g_{2,11}(x_1))}.
\end{equation}
Reasoning as in  \citep[Lemma A.1]{granero2019models}, we get
\begin{align*}
\w{v}(k,x_2)&=c_1(k)e^{|k|x_2}+c_2(k)e^{-|k|x_2}+\frac{1}{2|k|}\int_{0}^{x_2}\pare{e^{|k|(x_2-y_2)}-e^{-|k|(x_2-y_2)}}\w{f}(k,y_2)\,dy_2\nonumber\\
&=c_1(k)e^{|k|x_2}+c_2(k)e^{-|k|x_2}+\frac{1}{|k|}\int_{0}^{x_2}\sh{|k|(x_2-y_2)}\w{f}(k,y_2)\,dy_2.
\end{align*}
After imposing the boundary conditions, we find that $c_1,\,c_2$ verify
\begin{align*}
c_1(k)&=-c_2(k),\\
c_2(k)&=-\frac{(\sh{|k|})^{-1}}{2|k|}\int_{-1}^{0}\sh{|k|(1+y_2)}\w{f}(k,y_2)\,dy_2.
\end{align*}
This leads to
\begin{align*}
\w{v}(k,x_2)
&=\frac{1}{\av{k}}\frac{\sh{|k|x_2}}{\sh{|k|}}\int_{-1}^{0}\sh{|k|(1+y_2)}\w{f}(k,y_2)\,dy_2+\frac{1}{|k|}\int_{0}^{x_2}\sh{|k|(x_2-y_2)}\w{f}(k,y_2)\,dy_2.
\end{align*}
The definition of $v$  in \eqref{eq:v}  implies
\begin{align*}
\w{u}(k,x_2)
&=\frac{1}{\av{k}}\frac{\sh{|k|x_2}}{\sh{|k|}}\int_{-1}^{0}\sh{|k|(1+y_2)}\w{f}(k,y_2)\,dy_2+\frac{1}{|k|}\int_{0}^{x_2}\sh{|k|(x_2-y_2)}\w{f}(k,y_2)\,dy_2\\
&\q+\w{g_1}(k)+x_2(\w{g_1}(k)-\w{g_2}(k)),
\end{align*}
from which
\begin{align*}
u(x)&=\frac{1}{\sqrt{2\pi}}\sum_{k\in\mathbb{Z}}
\biggl\{
\frac{\sinh(|k|x_2)}{\sinh(|k|)} \frac{1}{|k|}\int_{-1}^0{\sinh(|k|(1+y_2))}\pare{w(y_2)\de_k+k^2\pare{\w{g_1}(k)+y_2\pare{\w{g_1}(k)-\w{g_2}(k)}}}\,dy_2 \nonumber \\
&\qq + \frac{1}{|k|}\int_{0}^{x_2}\sinh(|k|(x_2-y_2))\pare{w(y_2)\de_k+k^2\pare{\w{g_1}(k)+y_2\pare{\w{g_1}(k)-\w{g_2}(k)}}}\,dy_2\nonumber\\
& \qq+\w{g_1}(k)+x_2(\w{g_1}(k)-\w{g_2}(k))
\biggr\}e^{ikx_1},
\end{align*}
thanks to the definition of $f$ in \eqref{eq:f}. The above expression can be simplified observing that $\de_k=0$ if $k\ne0$, and $\de_0=1$. We hence compute
\begin{align*}
\lim_{k\to0}\frac{\sinh(|k|x_2)}{\sinh(|k|)} \frac{1}{|k|}\int_{-1}^0{\sinh(|k|(1+y_2))}w(y_2)\,dy_2&=x_2\int_{-1}^0(1+y_2)w(y_2)\,dy_2,
\\
\lim_{k\to0}\frac{1}{|k|}\int_{0}^{x_2}\sinh(|k|(x_2-y_2))w(y_2)\,dy_2&=\int_{0}^{x_2}(x_2-y_2)w(y_2)\,dy_2.
\end{align*}
 We thus recover \eqref{eq:u}.

We conclude computing
\begin{align*}
u,_1(x_1,0)&=\frac{1}{\sqrt{2\pi}}\sum_{k\in\mathbb{Z}} ik\w{g_1}(k)e^{ikx_1}=g_{1,1}(x_1),
\\
u,_1(x_1,-1)&=\frac{1}{\sqrt{2\pi}}\sum_{k\in\mathbb{Z}} ik \w{g_2}(k) e^{ikx_1}=g_{2,1}(x_1),
\end{align*}
and
\begin{align*}
\w{u},_2(k,x_2)&=\int_{-1}^0(1+y_2)w(y_2)\,dy_2+\int_{0}^{x_2}w(y_2)\,dy_2\\
&\q+\frac{\ch{|k|x_2}}{\sh{|k|}}k^2\int_{-1}^0{\sinh(|k|(1+y_2))}\pare{\w{g_1}(k)+y_2\pare{\w{g_1}(k)-\w{g_2}(k)}}\,dy_2 \\
&\q+k^2\int_{0}^{x_2}\cosh(|k|(x_2-y_2))\pare{\w{g_1}(k)+y_2\pare{\w{g_1}(k)-\w{g_2}(k)}}\,dy_2\\
&\q+\w{g_1}(k)-\w{g_2}(k),
\end{align*}
from which

\begin{align*}
u,_2(x_1,0)&=\int_{-1}^0(1+y_2)w(y_2)\,dy_2\\
&\q+\frac{1}{\sh{|k|}}k^2\int_{-1}^0{\sinh(|k|(1+y_2))}\pare{\w{g_1}(k)+y_2\pare{\w{g_1}(k)-\w{g_2}(k)}}\,dy_2 \\
&\q+\w{g_1}(k)-\w{g_2}(k),
\\
u,_2(x_1,-1)&=\int_{-1}^0y_2w(y_2)\,dy_2\\
&\q+\cth{\av{k}}k^2\int_{-1}^0{\sinh(|k|(1+y_2))}\pare{\w{g_1}(k)+y_2\pare{\w{g_1}(k)-\w{g_2}(k)}}\,dy_2 \\
&\q-k^2\int_{-1}^{0}\cosh(|k|(1+y_2))\pare{\w{g_1}(k)+y_2\pare{\w{g_1}(k)-\w{g_2}(k)}}\,dy_2\\
&\q+\w{g_1}(k)-\w{g_2}(k).
\end{align*}

\end{proof}

\begin{lemma}\label{lem:sol-ell2}
Consider the problem
\begin{align*}
\D u(x_1,x_2)&= w(x_1,x_2)&&\t{in }\mathbb{S}^1\times(-1,0),\\
u(x_1,0)&=g_1&&\t{for }x_1\in \mathbb{S}^1,\\
u(x_1,-1)&=g_2&&\t{for }x_1\in \mathbb{S}^1.
\end{align*}
Then
\begin{align}\label{eq:u2}
u(x_1,x_2)&=\frac{1}{\sqrt{2\pi}}\sum_{k\in\mathbb{Z}}
\left\{
\w{G}_1(k,x_2)+\w{G}_2(k,x_2) +\frac{\sinh(|k|x_2)}{\sinh(|k|)} \frac{1}{|k|}\int_{-1}^0{\sinh(|k|(1+y_2))}\w{w}(k,y_2)\,dy_2\right.\nonumber \\
&\left.\qq + \frac{1}{|k|}\int_{0}^{x_2}\sinh(|k|(x_2-y_2))\w{w}(k,y_2)\,dy_2
\right\}e^{ikx_1},
\end{align}
where
\begin{align*}
\w{G}_1(k,x_2)=e^{|k|x_2}\w{g_1}(k)+e^{-|k|}\frac{\sinh(|k|x_2)}{\sinh(|k|)} \w{g_1}(k)
\q\t{and}\q
\w{G}_2(k,x_2)=-\frac{\sinh(|k|x_2)}{\sinh(|k|)}\w{g_2}(k).
\end{align*}
Furthermore,
\begin{equation*}
u,_1(x_1,0)=g_{1,1}(x_1),\qq
u,_1(x_1,-1)=g_{2,1}(x_1),
\end{equation*}
and
\begin{align*}
u,_2(x_1,0)&= \Lambda(\sinh(\Lambda))^{-1}\bra{\cosh(\Lambda)g_1(x_1)-g_2(x_1)}  +(\sinh(\Lambda))^{-1}\int_{-1}^0\sinh(\Lambda(1+y_2))w(x_1,y_2)\,dy_2,
\\
u,_2(x_1,-1)
&=\Lambda (\sinh(\Lambda))^{-1}\bra{g_1(x_1)-\cosh(\Lambda)g_2(x_1)}  +\coth(\Lambda)\int_{-1}^0\sinh(\Lambda(1+y_2))w(x_1,y_2) \,dy_2,
\\
&\q -\int_{-1}^0\cosh(\Lambda(1+y_2)) w(x_1,y_2)\,dy_2.
\end{align*}
\end{lemma}
\begin{proof}
Reasoning as in  \citep[Lemma A.1]{granero2019models}, we get
\begin{equation*}\label{eq:71}
\w{u}(k,x_2)=c_1(k)e^{|k|x_2}+c_2(k)e^{-|k|x_2}+\frac{1}{2|k|}\int_{0}^{x_2}\pare{e^{|k|(x_2-y_2)}-e^{-|k|(x_2-y_2)}}\w{w}(k,y_2)\,dy_2.
\end{equation*}
After imposing the boundary conditions, we find that $c_1,\,c_2$ verify
\begin{align*}
c_1(k)&=\w{g_1}-c_2(k),\\
c_2(k)&= \frac{1}{e^{|k|}-e^{-|k|}} \bra{\w{g_2}-e^{-|k|}\w{g_1}+\frac{1}{2|k|}\int_{-1}^{0}\pare{e^{-|k|(1+y_2)}-e^{|k|(1+y_2)}}\w{w}(k,y_2)\,dy_2}.
\end{align*}
This leads to \eqref{eq:u2} and
\begin{align*}
\w{G}_1(k,x_2)=e^{|k|x_2}\w{g_1}(k)+\frac{e^{|k|x_2}-e^{-|k|x_2}}{e^{|k|}-e^{-|k|}}e^{-|k|} \w{g_1}(k)
\q\t{and}\q
\w{G}_2(k,x_2)=-\frac{e^{|k|x_2}-e^{-|k|x_2}}{e^{|k|}-e^{-|k|}} \w{g_2}(k).
\end{align*}

We conclude computing
\begin{align*}
u,_1(x_1,0)&=\frac{1}{\sqrt{2\pi}}\sum_{k\in\mathbb{Z}} ik\w{g_1}(k)e^{ikx_1},
\\
u,_1(x_1,-1)&=\frac{1}{\sqrt{2\pi}}\sum_{k\in\mathbb{Z}} ik \w{g_2}(k) e^{ikx_1},
\end{align*}
and
\begin{align*}
u,_2(x_1,0)&=\frac{1}{\sqrt{2\pi}}\sum_{k\in\mathbb{Z}}|k|
\left\{\frac{e^{|k|}+e^{-|k|}}{e^{|k|}-e^{-|k|}}\w{g_1}(k)-\frac{2}{e^{|k|}-e^{-|k|}}\w{g_2}(k)
\right.\\
&\left. \qq
-\frac{1}{2|k|}\frac{2}{e^{|k|}-e^{-|k|}}\int_{-1}^0\pare{e^{-|k|(1+y_2)}-e^{|k|(1+y_2)}}\w{w}(k,y_2)\,dy_2
\right\}e^{ikx_1},
\\
&=\frac{1}{\sqrt{2\pi}}\sum_{k\in\mathbb{Z}}|k|
\biggl\{\coth(|k|)\w{g_1}(k)-(\sinh(|k|))^{-1}\w{g_2}(k) \\
&\qq
-\frac{1}{2|k|}(\sinh(|k|))^{-1}\int_{-1}^0\pare{e^{-|k|(1+y_2)}-e^{|k|(1+y_2)}}\w{w}(k,y_2)\,dy_2
\biggl\}e^{ikx_1},
\\
u,_2(x_1,-1)&=\frac{1}{\sqrt{2\pi}}\sum_{k\in\mathbb{Z}}|k|
\left\{
\frac{2}{e^{|k|}-e^{-|k|}}\w{g_1}(k)-\frac{e^{|k|}+e^{-|k|}}{e^{|k|}-e^{-|k|}}\w{g_2}(k)\right.\\
&\left.\qq -\frac{1}{2|k|}\frac{e^{|k|}+e^{-|k|}}{e^{|k|}-e^{-|k|}}\int_{-1}^0\pare{e^{-|k|(1+y_2)}-e^{|k|(1+y_2)}}\w{w}(k,y_2)\,dy_2\right.\\
&\left.\qq -\frac{1}{2|k|}\int_{-1}^{0}\pare{e^{-|k|(1+y_2)}+e^{|k|(1+y_2)}}\w{w}(k,y_2)\,dy_2
\right\}e^{ikx_1},\\
&=\frac{1}{\sqrt{2\pi}}\sum_{k\in\mathbb{Z}}|k|
\biggl\{
(\sinh(|k|))^{-1}\w{g_1}(k)-\coth(|k|)\w{g_2}(k) \\
&\qq -\frac{1}{2|k|}\coth(|k|)\int_{-1}^0\pare{e^{-|k|(1+y_2)}-e^{|k|(1+y_2)}}\w{w}(k,y_2)\,dy_2 \\
&\qq -\frac{1}{2|k|}\int_{-1}^{0}\pare{e^{-|k|(1+y_2)}+e^{|k|(1+y_2)}}\w{w}(k,y_2)\,dy_2
\biggr\}e^{ikx_1}.
\end{align*}
Note that $\frac{|k|}{e^{|k|}-e^{-|k|}}=O(1)$ for $k\sim 0$.

\end{proof}

\section*{Acknowledgments}
The authors thank Alberto d'Onofrio for some stimulating discussions and comments.\\
Both authors are funded by  the project "An\'alisis Matem\'atico Aplicado y Ecuaciones Diferenciales" Grant PID2022-141187NB-I00 funded by MCIN /AEI /10.13039/501100011033 / FEDER, UE and acronym "AMAED". This publication is part of the project PID2022-141187NB-I00 funded by MCIN/ AEI /10.13039/501100011033.
 The work of M.M. was partially supported by a public grant as part of the Investissement d'avenir project, FMJH; by Grant RYC2021-033698-I, funded by the Ministry of Science and Innovation/State Research Agency/10.13039/501100011033 and by the European Union "NextGenerationEU/Recovery, Transformation and Resilience Plan"; the project PID2022-140494NA-I00 funded by MCIN/AEI/10.13039/501100011033/ FEDER, UE.

\bibliographystyle{abbrvnat}
\setcitestyle{authoryear,open={(},close={)}}
\bibliography{references}

@article{araujo2004history,
  title={A history of the study of solid tumour growth: the contribution of mathematical modelling},
  author={Araujo, Robyn P and McElwain, DL Sean},
  journal={Bulletin of mathematical biology},
  volume={66},
  number={5},
  pages={1039--1091},
  year={2004},
  publisher={Elsevier}
}

@article{byrne2010dissecting,
  title={Dissecting cancer through mathematics: from the cell to the animal model},
  author={Byrne, Helen M},
  journal={Nature Reviews Cancer},
  volume={10},
  number={3},
  pages={221},
  year={2010},
  publisher={Nature Publishing Group}
}

@article{friedman1999analysis,
  title={Analysis of a mathematical model for the growth of tumors},
  author={Friedman, Avner and Reitich, Fernando},
  journal={Journal of mathematical biology},
  volume={38},
  number={3},
  pages={262--284},
  year={1999},
  publisher={Springer}
}

@article{d2007nonlinear,
  title={A nonlinear mathematical model of cell turnover, differentiation and tumorigenesis in the intestinal crypt},
  author={d’Onofrio, Alberto and Tomlinson, Ian PM},
  journal={Journal of theoretical biology},
  volume={244},
  number={3},
  pages={367--374},
  year={2007},
  publisher={Elsevier}
}

@article{d2008metamodeling,
  title={Metamodeling tumor--immune system interaction, tumor evasion and immunotherapy},
  author={d’Onofrio, Alberto},
  journal={Mathematical and Computer Modelling},
  volume={47},
  number={5-6},
  pages={614--637},
  year={2008},
  publisher={Elsevier}
}

@article{wu2020asymptotic,
  title={Asymptotic behavior of a nonlinear necrotic tumor model with a periodic external nutrient supply.},
  author={Wu, Junde and Xu, Shihe},
  journal={Discrete \& Continuous Dynamical Systems-Series B},
  volume={25},
  number={7},
  year={2020}
}

@article{cui2002analysis,
  title={Analysis of a mathematical model for the growth of tumors under the action of external inhibitors},
  author={Cui, Shangbin},
  journal={Journal of mathematical biology},
  volume={44},
  number={5},
  pages={395--426},
  year={2002},
  publisher={Springer}
}

@article{zhou2009existence,
  title={Existence and asymptotic behavior of solutions to a moving boundary problem modeling the growth of multi-layer tumors},
  author={Zhou, Fujun and Wu, Junde and Cui, Shangbin},
  journal={Comm. Pure Appl. Anal},
  volume={8},
  pages={1669--1688},
  year={2009}
}

@article{d2009fractal,
  title={Fractal growth of tumors and other cellular populations: linking the mechanistic to the phenomenological modeling and vice versa},
  author={d’Onofrio, Alberto},
  journal={Chaos, Solitons \& Fractals},
  volume={41},
  number={2},
  pages={875--880},
  year={2009},
  publisher={Elsevier}
}

@article{fiore2020mechanics,
  title={Mechanics of a multilayer epithelium instruct tumour architecture and function},
  author={Fiore, Vincent F and Krajnc, Matej and Quiroz, Felipe Garcia and Levorse, John and Pasolli, H Amalia and Shvartsman, Stanislav Y and Fuchs, Elaine},
  journal={Nature},
  volume={585},
  number={7825},
  pages={433--439},
  year={2020},
  publisher={Nature Publishing Group UK London}
}

@book{cristini2010multiscale,
  title={Multiscale modeling of cancer: an integrated experimental and mathematical modeling approach},
  author={Cristini, Vittorio and Lowengrub, John},
  year={2010},
  publisher={Cambridge University Press}
}

@article{greenspan1972models,
  title={Models for the growth of a solid tumor by diffusion},
  author={Greenspan, HP},
  journal={Studies in Applied Mathematics},
  volume={51},
  number={4},
  pages={317--340},
  year={1972},
  publisher={Wiley Online Library}
}

@article{greenspan1974self,
  title={On the self-inhibited growth of cell cultures.},
  author={Greenspan, HP},
  journal={Growth},
  volume={38},
  number={1},
  pages={81},
  year={1974}
}

@article{greenspan1976growth,
  title={On the growth and stability of cell cultures and solid tumors},
  author={Greenspan, HP},
  journal={Journal of theoretical biology},
  volume={56},
  number={1},
  pages={229--242},
  year={1976},
  publisher={Elsevier}
}

@article{granero2025nonlocal,
  title={A nonlocal equation describing tumor growth},
  author={Granero-Belinch{\'o}n, Rafael and Magliocca, Martina},
  journal={Mathematical Models and Methods in Applied Sciences},
  volume={35},
  number={03},
  pages={585--609},
  year={2025},
  publisher={World Scientific}
}

@article{byrne1995growth,
  title={Growth of nonnecrotic tumors in the presence and absence of inhibitors},
  author={Byrne, HM and Chaplain, M A\_ J},
  journal={Mathematical biosciences},
  volume={130},
  number={2},
  pages={151--181},
  year={1995},
  publisher={Elsevier}
}

@article{granero2019models,
  title={Models for damped water waves},
  author={Granero-Belinch{\'o}n, Rafael and Scrobogna, Stefano},
  journal={SIAM Journal on Applied Mathematics},
  volume={79},
  number={6},
  pages={2530--2550},
  year={2019},
  publisher={SIAM}
}

\end{document}